\documentclass[a4paper, reqno]{amsart}

\usepackage{geometry}
\usepackage[all]{xy}
\usepackage{verbatim}
\usepackage{amstext}
\usepackage{amsmath,amssymb}
\usepackage{amscd}
\usepackage{amsthm}
\usepackage{amsfonts}
\usepackage{enumerate}
\usepackage{comment}
\usepackage{graphicx}
\usepackage{latexsym}
\usepackage{mathrsfs}
\usepackage{mathtools}
\xyoption{all}
\usepackage{pstricks}
\usepackage{lscape}
\usepackage{comment}
\renewcommand{\a}{ {\bf a}}
\newcommand{\e}{ {\bf e}}
\newtheorem{theorem}{Theorem}[section]

\newtheorem{corollary}[theorem]{Corollary}
\newtheorem{lemma}[theorem]{Lemma}
\newtheorem{proposition}[theorem]{Proposition}
\newtheorem{definition-proposition}[theorem]{Definition-Proposition}

\theoremstyle{definition}
\newtheorem{definition}[theorem]{Definition}
\newtheorem{assumption}[theorem]{Assumption}
\newtheorem{remark}[theorem]{Remark}
\newtheorem{example}[theorem]{Example}
\newtheorem{observation}[theorem]{Observation}

\newcommand{\Liota}{\iota_\lambda}
\newcommand{\Riota}{\iota_\rho}
\newcommand{\Lpi}{\pi_\lambda}
\newcommand{\Rpi}{\pi_\rho}
\newcommand{\Ldelta}{\delta_\lambda}

\newcommand{\LDelta}{\Delta_\lambda}
\newcommand{\RDelta}{\Delta_\rho}

\newcommand{\Ares}[1]{\mathcal{A}_{#1}}

\newcommand{\id}{\mathsf {id}}
\newcommand{\AAA}{\A[\eta_1^{\frac{1}{p_1}},\dots,\eta_n^{\frac{1}{p_n}}]}

\newcommand{\grid}{S}
\renewcommand{\O}{{\mathcal O}}%sorry I am used to this - Boris
\newcommand{\F}{\mathcal{F}}
\newcommand{\E}{\mathcal{E}}
\newcommand{\A}{\mathcal{A}}

\renewcommand{\L}{\mathbf{L}}
\newcommand{\Z}{\mathbb{Z}}
\renewcommand{\H}{\mathbb{H}}

\newcommand{\matt}{\begin{bmatrix}R&I&I\\R&R&I\\R&R&R\end{bmatrix}}
\newcommand{\mat}{\begin{bmatrix}k&0\\k&k\end{bmatrix}}

\newcommand{\Ext}{\operatorname{Ext}\nolimits}

\newcommand{\Hom}{\operatorname{Hom}\nolimits}
\newcommand{\End}{\operatorname{End}\nolimits}

\newcommand{\RGamma}{\mathbf{R}\Gamma}
\newcommand{\RHOM}{\mathbf{R}\mathcal Hom}
\newcommand{\RHom}{\mathbf{R}\strut\kern-.2em\operatorname{Hom}\nolimits}
\newcommand{\cok}{\operatorname{cok}\nolimits}

\renewcommand{\P}{\mathbb{P}}

\newcommand{\Ae}{\frac{A}{\langle e\rangle}}

\newcommand{\DDD}{\mathsf{D}}
\newcommand{\bo}{\operatorname{b}\nolimits}
\DeclareMathOperator{\thick}{\mathsf{thick}}
\DeclareMathOperator{\Pic}{\mathsf{Pic}}
\DeclareMathOperator{\coh}{\mathsf{coh}}
\DeclareMathOperator{\moduleCategory}{\mathsf{mod}} \renewcommand{\mod}{\moduleCategory}

\DeclareMathOperator{\gldim}{gldim}
\DeclareMathOperator{\injdim}{inj.dim}

\begin{document}
\title{A recollement approach to Geigle-Lenzing weighted projective varieties}

\author[Lerner]{Boris Lerner}
\address{B. Lerner: School of Mathematics and Statistics, UNSW, Sydney, 2052, Australia.}
\email{boris@unsw.edu.au}
\urladdr{http://www.unsw.edu.au/~borislerner}

\author[Oppermann]{Steffen Oppermann}
\address{S. Oppermann: Institutt for matematiske fag, NTNU, 7491 Trondheim, Norway}
\email{steffen.oppermann@math.ntnu.no}
\urladdr{http://www.math.ntnu.no/~opperman/}
\thanks{Most of the work on this paper was done while the second author visited Nagoya. He would
like to thank Osamu Iyama and his group for their hospitality. The first author thanks Daniel Chan
for his valuable comments and suggestions. Both authors are deeply in dept to
Osamu Iyama for providing many ideas, suggestions,  and comments to this project.}

\thanks{B.~L.\ was partially supported by JSPS postdoctoral fellowship program. S.~O.\ was supported by NRF grant 221893.}

%\thanks{{\em Key words and phrases.} Auslander-Reiten theory, preprojective algebra, Fano algebra, $n$-representation finite algebra, representation dimension}

\begin{abstract}
We introduce a new method for expanding an abelian category and study it using recollements. In
particular, we give a criterion for the 
existence of cotilting objects. We show, using techniques from noncommutative algebraic geometry, that our construction encompasses the category of
coherent sheaves on Geigle-Lenzing weighted projective lines. We apply our construction to some concrete examples and obtain new
weighted projective varieties and analyse the endomorphism algebras of their tilting bundles.

\end{abstract} 

\maketitle
\section{Introduction}\label{sec:introduction}

In their famous paper \cite{GL} Geigle and Lenzing introduced an important class of abelian categories with a tilting object (see Definition
\ref{def:tilting}) which have subsequently been called
\emph{coherent sheaves on Geigle-Lenzing (GL) weighted projective lines}. This category has played an important role in many fields, in
particular representation theory of finite dimensional algebras. It was recently generalised in \cite{HIMO} to include higher
dimensional projective spaces. A different interpretation of these categories was discovered in \cite{CI, RVdB} for the dimension $1$ case
and more generally in \cite{IL}, where these categories are shown to be equivalent to module categories $\mod A$ of a certain order $A$ on $\P^d$, which we call a
\emph{GL order} (see below).
Viewing GL weighted projective spaces as module categories allows for further,
very fruitful,
generalisations which is what we explore in this paper. The
idea is rather simple: in \cite{IL} all GL orders that were considered were always sheaves on $\P^d$, 
now we allow the centre to be other varieties:

\begin{definition}\label{def:order}
	Fix a scheme $X$ over a field $k$
	and for $i=1,\dots,n$ fix prime divisors $L_i$ on $X$ and integer
	weights $p_i\ge 2$. A \emph{GL order $A$}
	with \emph{centre} $X$ associated to this data
	is a sheaf of noncommutative algebras of the form \[A=\bigotimes_{i=1}^nA_i,
		\ \mbox{ where }\ \quad A_i:=
		\begin{bmatrix}
			\O&\O(-L_i)&\dots&\O(-L_i)&\O(-L_i)\\
			\O&\O&\dots&\O(-L_i)&\O(-L_i)\\
			\vdots&\vdots&\ddots&\vdots&\vdots\\
			\O&\O&\dots&\O&\O(-L_i)\\
			\O&\O&\dots&\O&\O
		\end{bmatrix}\subset M_{p_i}(\O)
	\] and $\O=\O_X$.
\end{definition}

The aim of this paper is to study the category $\mod A$ of GL orders $A$,
in particular, we give a criterion on the existence of tilting sheaves.
First we give a description of $\mod A$ in terms of \emph{grid categories}
$\AAA$, which are constructed from an abelian category $\A$ with endofunctors
$F_i\colon\A\to\A$ and natural transformations $\eta_i\colon F_i\to\id_{\A}$ 
for $i=1,\ldots,n$.
Moreover we give a sufficient condition for $\AAA$ to have a tilting object.
Then we apply these results to GL orders and obtain the following: for each subset
$I\subseteq \{1,\dots,n\}$ we denote by \[\O_I=\bigotimes_{i\in I}\O_{L_i}.\] 
\begin{theorem}[Theorem \ref{Thm:order}]Let $A$ be a GL order on a smooth projective variety over an
	algebraically closed field $k$ and suppose $\sum L_i$ is a simple normal crossings divisor.
	Assume there is a collection of
	tilting objects $T_I\in \mod \O_I$  for all $I\subseteq \{1,\dots,n\}$, such that
\begin{itemize}
\item $T_I\otimes \O_J\otimes \O(-L_j)\to T_I\otimes \O_J$ is injective whenever $I$, $J$, and $\{j\}$ are pairwise disjoint;
\item $\Ext^i_{\O_{J}}(T_I \otimes \O_J, T_{I \cup J})=0$ for all $i>0$, whenever $I  \cap J = \emptyset$.
\end{itemize}
then 
	\[
		\bigoplus_{I\subseteq \{1,\dots,n\}}\left(\bigotimes_{i\not\in I}
		A_if_i	
		\otimes
		\bigotimes_{i\in I}
		\frac{A_i}{\langle e_i\rangle}
		\otimes T_I\right)
	\] is tilting in $\mod A$, where $e_i$ and $f_i$ are matrices of size $p_i\times p_i$
	with $1$ in the bottom right (respectively top left) position, and
	$0$'s elsewhere.
\end{theorem}

We apply this result to several concrete projective varieties.
For instance, let $X=\P^1\times \P^1$ and  $L_i\sim(1,1)$ for
$i=1,2$. Suppose $L_1\cap L_2=p+q$. Consider \[A=
	\begin{bmatrix}
		\O&\O(-L_1)\\\O&\O
	\end{bmatrix}\otimes
	\begin{bmatrix}
		\O&\O(-L_2)\\\O&\O
	\end{bmatrix}
\]Then $T_\emptyset=\O\oplus \O(1,0)\oplus \O(0,1)\oplus
\O(1,1)$, $T_{i}=\O_{L_i}(1)\oplus\O_{L_i}(2)$ for
$i=1,2$ and $T_{1,2}=\O_{p}\oplus\O_q$ satisfies the
assumptions of the theorem. Hence a tilting object in
$\mod A$ is\[
	\left(A_1f_1\otimes A_2f_2\otimes
	T_{\emptyset}\right)\oplus
	\left(\frac{A_1}{\langle e_1\rangle}\otimes
	A_2f_2\otimes T_{1}\right)\oplus \left(A_1f_1\otimes
	\frac{A_2}{\langle e_2\rangle}\otimes T_{2}\right)\oplus
	\left(\frac{A_1}{\langle e_1\rangle}\otimes
	\frac{A_2}{\langle e_2\rangle }\otimes
	T_{1,2}\right)
\]

Our approach is rather general and categorical. In Section \ref{sec:setup} we begin with an abelian category $\A$ and an integer $n\ge 1$, we fix
endofunctors $F_i$, natural transformations $\eta_i$ and integer weights $p_i\geq 2$ and construct, a new category $\AAA$. 
In Section \ref{1weight} we analyse the case $n=1$ and, using recollements, give a criterion for this category to have a cotilting
object. Our emphasis on cotilting, as opposed to tiling, is because the cotilting criterion is easier to check in $\mod A$, than the
corresponding tilting criterion, which can also be easily derived using a similar approach to ours. However, due to the existence of Serre duality 
in $\mod A$, cotilting and tilting are actually equivalent and so this subtlety causes no issues in practice.
In Section \ref{sec:general} we analyse the situation for an arbitrary $n$. In Section
	\ref{sec:global} the global dimension of these categories is computed, showing that it often coincides with the global dimension of the original category. In Section~\ref{sec:order} we translate the categorical results to orders, to obtain the main result as stated above. Finally, in
Section \ref{sec:applications} we show how our results may be applied to concrete situations: to Hirzebruch surfaces and to projective
spaces. In the $\P^d$ case we show that the tilting bundle we obtain is in fact a generalisation of the \emph{squid} algebra.

\section{Setup and Notation}\label{sec:setup} 
Throughout $k$ denotes an algebraically closed field. Let $\A$
be a $k$-linear, $\Hom$-finite, abelian category. Throughout, we compose morphism left to right. 
Fix for  $i=1,\dots,n$,
commuting exact functors $F_i\colon\A\to\A$, natural transformations 
$\eta_i\colon F_i\to\id_\A$ and integer weights $p_i\ge 2$.
For any $M\in\A$, we denote by
 \[\eta_i(M)\colon F_iM\to M\] instead
of the more conventional notation ${\eta_{i,M}}$.
Using this data, we will now
define a new category \[\AAA\]
of $n$-dimensional grids  of size $(p_1+1)\times\dots\times (p_n+1)$ of commuting morphisms. To make this precise we need to introduce 
some notation:
Let \[ \grid = \{1, \ldots, p_1\} \times \cdots \times \{1, \ldots, p_n\} \subseteq \Z^n, \]
and denote by $\e_i$ the $i$-th basis vector in $\Z^n$. 

Throughout, to allow for compact notation, whenever objects or morphisms are indexed by $\grid$ we also allow non-positive indices and interpret them via $M_{\a} := F_i M_{\a + p_i \e_i}$ and similar for morphisms. Note that the assumption that the $F_i$ commute makes this well-defined even if several indices are non-positive.

With this notation, we define objects of $\AAA$ to be tuples
\[( (M_\a)_{\a \in \grid},  (f^i_\a \colon M_{\a-\e_i} \to M_{\a})_{\substack{ 1\le i\le n \\ \a \in \grid}} ), \]
subject to the conditions:
\begin{itemize}

	\item (commutativity condition) for any $i, j \in \{1, \ldots, n\}$ and $\a \in \grid$ we have $ f^j_{\a - \e_i} f^i_{\a}=
	f^i_{\a - \e_j}f^j_{\a} $ i.e. the following diagram commutes:
\[\xymatrix{
	M_{\a-\e_i-\e_j} \ar[r]^{f^j_{\a-\e_i}}\ar[d]_{f^i_{\a-\e_j}}&M_{\a-\e_i}\ar[d]^{f^i_{\a}}\\
	M_{\a-\e_j} \ar[r]^{f^j_\a}&M_{\a}
}\]

\item (cycle condition)  for any $i \in \{1, \ldots, n\}$ and $\a \in \grid$ we have 
	 $f_{\a - (p_i - 1) \e_i}^i \cdots f_{\a-\e_i}^i f_{\a}^i=\eta_i(M_{\a})$.
\end{itemize}

A morphism $\varphi \colon (M_\a,f^i_\a) \to (N_\a,g^i_\a) \in\AAA$ is a set of morphisms $\varphi_\a\colon M_\a\to N_\a$ in $\A$ with $\a \in \grid$, such that  the following diagram commutes:
\[\xymatrix{
	M_{\a-\e_i} \ar[r]^{f^i_{\a}}\ar[d]_{\varphi_{\a-\e_i}}&M_{\a}\ar[d]^{\varphi_{\a}}\\N_{\a-\e_i} \ar[r]^{g^i_\a}&N_{\a}
}\]

\begin{example}\label{ex:1}
	If $n=1$ then objects in $A[\eta^{\frac{1}{p}}]$ are
	sequences\[M_0=FM_p\xrightarrow{f_1}M_1\xrightarrow{f_2}\dots\xrightarrow{f_{p}}M_p\]
	such that the composition 
	\[FM_{p-d-1}\xrightarrow{Ff_{p-d}}\dots\xrightarrow{Ff_{p}}FM_p\xrightarrow{f_1}M_1\xrightarrow{f_2}\dots\xrightarrow{f_{p-d}}M_{p-d}\]
	is equal to $\eta(M_{p-d})\colon FM_{p-d}\to M_{p-d}$ for all $0\le d\le p-1$.
\end{example}
\begin{example}
	Suppose $n=2,\; p_1=2$ and $p_2=3$. Then objects in $\A[\eta_1^{\frac{1}{p_1}}, \eta_2^{\frac{1}{p_2}}]$ are
	\[\xymatrix{
		F_1F_2M_{2,3}\ar[r]^{F_1f^2_{(2,1)}}\ar[d]_{F_2f^1_{(1,3)}}&F_1
		M_{2,1}\ar[r]^{F_1f^2_{(2,2)}}\ar[d]_{f^1_{(1,1)}}&F_1M_{2,2}\ar[r]^{F_1f^2_{(2,3)}}\ar[d]_{f^1_{(1,2)}}&
		F_1M_{2,3}\ar[d]_{f^1_{(1,3)}}\\
		F_2M_{1,3}\ar[r]^{f^2_{(1,1)}}\ar[d]_{F_2f^1_{(2,3)}}&M_{1,1}\ar[r]^{f^2_{(1,2)}}\ar[d]_{f^1_{(2,1)}}&M_{1,2}\ar[r]^{f^2_{(1,3)}}\ar[d]_{f^1_{(2,2)}}&M_{1,3}\ar[d]_{f^1_{(2,3)}}\\
		F_2M_{2,3}\ar[r]^{f^2_{(2,1)}}&M_{2,1}\ar[r]^{f^2_{(2,2)}}&M_{2,2}\ar[r]^{f^2_{(2,3)}}&M_{2,3}
	}\] where all the squares commute and the rows and columns satisfy the cycle
	conditions.
\end{example}
In this paper, we are primarily concerned with the existence a cotilting objects in $\AAA$.
\begin{definition}\label{def:tilting}
	Let $\A$ be an abelian category. We say an object a in $T$ is \emph {tilting} (resp.
	\emph {cotilting}), if satisfies the following $2$
	conditions:
	\begin{itemize}
		\item Rigidity: $\Ext^i_\A(T,T)=0$ for all
			$i>0$,
		\item Generation (resp. cogeneration): 
			$\Ext^i_\A(T,M)=0$ (resp.
			$\Ext^i_\A(M,T)=0$) for all
			$i\ge 0$ implies $M=0$.
	\end{itemize}
\end{definition}
In the next $2$ sections will focus on proving results regarding cotilting objects, rather than tilting. Analogous results can be derived
for the latter, however the corresponding results, are of little practical use in the applications to orders which we
have in mind in Sections \ref{sec:order} and \ref{sec:applications}. However, due to the existence of Serre duality in the order setting, tilting
and cotilting objects coincide.
\section{Cotilting for the case with only one weight}\label{1weight}
In this section we analyse the situation where $n=1$, i.e. the category $\A[\eta^{\frac{1}{p}}]$. Recall that this category was already
introduced in Example \ref{ex:1}. The results we obtain, will be useful when we study the general case.

We denote by $\Ares\eta$ the full subcategory of $\A$ with objects given by:
\[\Ares\eta:=\{M\in\A\mid\eta(M)=0\}\]

We begin, with a further simplification, namely assume that $F=0$. In this case, \[\A[0^{\frac{1}{p}}]={\rm rep}_{\A}A_p:={\rm Fun}(A_p,
\A)\] where $A_p$ is the linearly orientated quiver of Dynkin type $A$ and $p$ vertices and viewed as a (finite) category in the obvious way.
We have an exact functor $\delta\colon {A}\to{\A}[0^{\frac{1}{p}}]$ with
\[\delta(M):= (0\to M\to\dots\to M)\oplus\dots\oplus (0\to M\to0\to\dots\to 0)\]
which has an exact left adjoint
\begin{align*}
\Ldelta(0\to M_1\to\dots\to M_{p})&=M_1\oplus\dots\oplus M_p\\
\end{align*}

\begin{lemma}\label{lemma:ext}
	Let $M,N\in\A$. Then $\Ext^i_{\A}(M,N)=0$ if and only if $\Ext^i_{\A[\eta^{\frac{1}{p}}]}(\delta M, 
	\delta N)=0$ for $i\ge 1$.
\end{lemma} 

	\begin{proof}
		Since $\delta$ and $\Ldelta$ are exact
	\[\Ext^i_{\A[\eta^{\frac{1}{p}}]}(\delta M,\delta N)=\Ext^i_{\A}(\Ldelta\delta
M,N)=\Ext^i_{\A}(M^n,N)\]
with $n=\frac p2(p+1)$. Note that the first equality follows from Lemma \ref{lem.exact_adjoints}
ahead.
	\end{proof}
\begin{proposition}\label{prop:tiltingobject}
	If $T$ is a cotilting object in $\A$  then $\delta(T)$ is a cotilting object in $\A[0^{\frac{1}{p}}]$.
\end{proposition}
\begin{proof}
	By Lemma \ref{lemma:ext} and the fact that $T$ is cotitling in $\A$ we see that $\delta(T)$ is rigid.
	We now prove that $\delta(T)$ cogenerates $\A[\eta^{\frac{1}{p}}]$. Let $M\in \A[0^{\frac{1}{p}}]$ and
	suppose $\Ext^i_{\A[0^{\frac{1}{p}}]}(M, \delta(T))=0$.
	Then\[0=\Ext^i_{\A[0^{\frac{1}{p}}]}(M,\delta T)=\Ext^i_\A(\Ldelta M,T)=\Ext^i_\A(M_1\oplus\dots\oplus M_p,T) \] and so, since
	$T$ cogenerates $\A$,
	$M_1=\dots=M_p=0$ i.e. $M=0$. 
\end{proof}
More generally, if $n=1$  we can analyse $\A[\eta^{\frac{1}{p}}]$, and in the next section (for an arbitrary $n$) $\AAA$,
using recollements, which we now define.
\begin{definition}[\cite{BBD}]
	Let $\A',\A,\A''$ be abelian categories. A \emph{recollement} is the following diagram of additive functors
\[\xymatrix@C=70pt{
{\A'} \ar[r]^{\iota}
& \A \ar@/^1.5pc/[l]_{\Riota} \ar@/_1.5pc/[l]_{\Liota}  \ar[r]^{\pi} 
& \A''
\ar@/_1.5pc/[l]_{\Lpi} \ar@/^1.5pc/[l]_{\Rpi}
}\]such that
\begin{enumerate}
	\item $(\Liota,\iota,\Riota)$ and $(\Lpi,\pi,\Rpi)$ are adjoint triples.
	\item $\iota, \Lpi$ and $\Rpi$ are fully faithful.
	\item ${\rm im}\;\iota=\ker\pi$.
\end{enumerate}
\end{definition} 
\begin{example}\label{eg:standard} Let $A$ be a ring, and $e$ an idempotent. Denote by $\mod A$ the category of left $A$-modules.
 Then we have the following recollement:
\[\xymatrix@C=70pt{
	\mod \Ae \ar[r]^{\iota}
& \mod A \ar@/^1.5pc/[l]_{\Riota} \ar@/_1.5pc/[l]_{\Liota}  \ar[r]^{\pi} 
& \mod eAe
\ar@/_1.5pc/[l]_{\Lpi} \ar@/^1.5pc/[l]_{\Rpi}
}\]
where 
	\begin{align*}
	\iota = {\rm inclusion},\quad \pi&=e(-),\quad \Lpi=Ae\otimes_{eAe} -, \quad \Rpi=\Hom_{eAe}(eA, -)\\
	\Liota&= \frac{A}{\langle e\rangle}\otimes_A -,\quad \Riota= \Hom_A(\frac{A}{\langle e\rangle},-)
\end{align*} If \[A=
	\begin{bmatrix}
		R&I&I\\R&R&I\\R&R&R
	\end{bmatrix}\ni
e=
\begin{bmatrix}
	0&0&0\\0&0&0\\0&0&1
\end{bmatrix}
\] where $R$ is a commutative ring and $I$ is a maximal ideal.
Then setting $k=R/I$ we see that\[\frac{A}{\langle e\rangle}=
	\begin{bmatrix}
		k&0&0\\k&k&0\\0&0&0
	\end{bmatrix}
\] and so the recollement becomes
\[\xymatrix@C=70pt{
	\mod \mat
	\ar[r]^{\iota}
& \mod \matt \ar@/^1.5pc/[l]_{\Riota} \ar@/_1.5pc/[l]_{\Liota}  \ar[r]^{\pi} 
& \mod R
\ar@/_1.5pc/[l]_{\Lpi} \ar@/^1.5pc/[l]_{\Rpi}
}\]
with 
\[
	\Lpi=
	\begin{bmatrix}
		I\\I\\R
	\end{bmatrix}\otimes_R -\quad
	\Rpi=
	\begin{bmatrix}
		R\\R\\R
	\end{bmatrix}\otimes_R -
\]
\end{example} We now return to the category $\A[\eta^{\frac{1}{p}}]$. Recall that we denote by $\A_{\eta}$ the subcategory of $\A$ consisting of all objects such that $\eta$ vanishes.

\begin{proposition} \label{prop.recollement}
	The following is a recollement:
\[\xymatrix@C=70pt{
	{\Ares\eta[0^{\frac{1}{p-1}}]} \ar[r]^{\iota}
	& \A[\eta^{\frac{1}{p}}] \ar@/^1.5pc/[l]_{\Riota} \ar@/_1.5pc/[l]_{\Liota}  \ar[r]^{\pi} 
& \A
\ar@/_1.5pc/[l]_{\Lpi} \ar@/^1.5pc/[l]_{\Rpi}
}\]
where the functors are defined by the following: 
\begin{align*}
	&\iota(0\to M_1\to\dots\to M_{p-1})=(0\to M_1\to\dots\to M_{p-1}\to 0)\\
	&\pi(FM_p\to M_1\to\dots\to M_p)=M_p\\
	&{\Lpi}M=(FM\xrightarrow{\id}FM\to\dots\to FM\xrightarrow{\eta(M)}M)\\
	&\Rpi M=(FM\xrightarrow{\eta(M)}M\xrightarrow{\id}M\to\dots\to M)\\
	&{\Liota}(FM_p\xrightarrow{f_1}M_1\xrightarrow{f_2}\dots\xrightarrow{f_{p}}M_p)=
	(0\to\cok f_1\to\cok{f_1f_2}\to\dots\to\cok{f_{1}\dots f_{p-1}})\\
	&{\Riota}(FM_p\xrightarrow{f_1}M_1\xrightarrow{f_2}\dots\xrightarrow{f_{p}}M_p)=
	(0\to\ker f_{2}\dots f_p\to\ker f_{3}\dots f_{p}\to\dots\to \ker f_{p})
\end{align*}
\end{proposition}
\begin{proof}
	Straightforward.
\end{proof}

We need to control how these functors affect $\Ext$-spaces. In the following, we prove
that exact adjoint functors are also adjoint with respect to $\Ext$.
Since the most usual way to see this is to use projective or injective resolutions, which we do not
assume to exist here, we give a small argument using Yoneda-extension groups.

\begin{lemma} \label{lem.exact_adjoints}
Let $\A$ and $\mathcal{B}$ be exact categories, $L \colon \A \to \mathcal{B}$ and $R \colon \mathcal{B} \to \A$ a pair of exact adjoint functors. Then
\[ \Ext_{\A}^n(A, R B) \simeq \Ext_{\mathcal{B}}^n(L A, B) \]
functorial in $A \in \A$ and $B \in \mathcal{B}$.
\end{lemma}

\begin{proof}
	We view the $\Ext$ groups as Yoneda-$\Ext$ groups and use the following notation: given
	$\mathbb{E}\in \Ext^i(A,B)$ and maps $\alpha\colon B\to B'$ and $\beta\colon A'\to A$ we denote by 
	$\alpha_*\mathbb{E}\in \Ext^i(A, B')$ the extension obtain by taking the pushout along $\alpha$ and by
	$\beta^* \mathbb{E}\in
	\Ext^i(A',B)$ the extension obtained by taking the pullback along $\beta$. Note that by
	\cite[Lemma 3.1.6]{ML} we have \[\alpha_*\beta^*\mathbb{E}\simeq \beta^*\alpha_*\mathbb{E}\]

To prove the lemma, we give two maps, and show that they are mutually inverse to each other.
From left to right, let $\mathbb{E} \in \Ext_{\A}^n(A, R B)$. Since $L$ is exact we may apply it to
$\mathbb{E}$, obtaining $L(\mathbb{E}) \in \Ext_{\mathcal{B}}^n(LA, LRB)$. Now consider the counit
of the adjunction $\varepsilon_B \colon LR B \to B$. Taking the pushout along this map we obtain
${\varepsilon_B}_*L(\mathbb{E}) \in \Ext_{\mathcal{B}}^n(LA, B)$.

Conversely, from right to left, we send the extension $\mathbb{E} \in \Ext_{\mathcal{B}}^n(LA, B)$
to $\omega_A^* R(\mathbb{E}) \in \Ext_{\A}^n(A, R B)$, where $\omega_A \colon A \to RLA$ denotes the unit of the adjunction.

Both constructions are well-defined on the Yoneda-extension groups, and are functorial. 
It remains to see that they are mutually inverse. Here we check that going from left to 
right and then back again one obtains the extension one started with. Checking that this
also works the other way around is very similar.

So let $\mathbb{E}  \in \Ext_{\A}^n(A, R B)$. Applying $L$ to $\mathbb{E}$ and
sending it to $\Ext_{\mathcal{B}}^n(L A, B)$ via a pushout along $\varepsilon_B$ 
and then applying $R$ and sending it back 
back to $\Ext_{\A}^n(A, R B)$ via the pullback along $\omega_A$
we obtain $\omega_A^* R({\varepsilon_B}_*L(\mathbb{E})) \in \Ext_{\A}^n(A, R B)$. Since $R$ is exact it commutes with pushouts,
so this is the same as applying $RL$ to $\mathbb{E}$ and then taking the pushout along
$R(\varepsilon_B)$ followed by a pullback along $\omega_A$. Thus we have: 
\[ \omega_A^* R({\varepsilon_B}_*L(\mathbb{E})) \simeq \omega_A^*R(\varepsilon_B)_* RL(\mathbb{E})\simeq
R(\varepsilon_B)_* \omega_A^*RL(\mathbb{E}). \]
Moreover, since $\omega$ is a natural transformation $\id \to RL$, we have that $\omega_A^*
RL(\mathbb{E}) \simeq {\omega_{RB}}_*\mathbb{E}$. Thus
\[  \omega_A^* R({\varepsilon_B}_*L(\mathbb{E}))\simeq
R(\varepsilon_B)_*{\omega_{RB}}_*\mathbb{E}. \]
Now the proof is completed using the general fact for adjoint pairs, that $ R(\varepsilon_B)
\circ \omega_{RB} = \id_{RB}$.
\end{proof}

Now we observe that in the recollement of Proposition~\ref{prop.recollement} the functors
$\iota,\pi,{\Lpi}$ and $\Rpi$ are all exact. In particular Lemma~\ref{lem.exact_adjoints} implies that
\[\Ext^i(M, \pi N)\simeq \Ext^i(\Lpi M, N),\quad {\rm and}\quad {\Ext^i(\pi M, N)}
\simeq \Ext^i( M, \Rpi N)\] for all $i\ge 0$.

The situation is slightly more involved for $\iota$, since none of the functors $\Liota$ or $\Riota$ is exact. To be able to still control its effect on $\Ext$-spaces we will need the following assumption.
\begin{assumption}
	For the remainder of this section, we assume that $\A$ has enough objects $M$ such that $\eta(M)$ is a
	monomorphism. I.e.\ for all objects $X\in \A$ there exists an object $M\in\A$ and a
	surjection $M\twoheadrightarrow X$ such that $\eta(M)$ is a monomorphism.
\end{assumption}

\begin{lemma}\label{lemma:resolving} With the above assumption we the subcategory of $\A$ given by:
	
\[\E:=\{FM_p\xrightarrow{f_1}M_1\xrightarrow{f_2}\dots\xrightarrow{f_{p}}M_p\in\A[\eta^{\frac{1}{p}}]\;|\;f_{1}\dots
f_{p-1}
\;\rm is\;a\;monomorphism\}\] is a resolving subcategory.
\end{lemma}
The reason for choosing this particular subcategory is because ${\rm im}\; \iota\subseteq \E$, a fact that
we will later need.
\begin{proof}
	Let $M^\bullet=FM_p\xrightarrow{f_1}M_1\xrightarrow{f_2}\dots\xrightarrow{f_{p}}M_p \in
	\A[\eta^{\frac{1}{p}}]$. For each $i=1,\dots,p$ let \[
		M_i^\bullet=FM_i\to FM_i\to\dots FM_i\xrightarrow{\eta(M_i)} M_i\to\dots \to M_i
	\] where the $\eta(M_i)$ is the $i$'th arrow from the right. 
	Note that we have a surjective map $\oplus M_i^\bullet\twoheadrightarrow M^\bullet$.
	Furthermore, by the assumption on $\A$, for all $i$, there exists an $X_i\in\A$ such
	that $X_i\twoheadrightarrow M_i$ and $\eta(X_i)$ is a monomorphism. Since $F$ is exact,
	$F(X_i)\twoheadrightarrow F(M_i)$ and so
	\[X_i^\bullet:=FX_i\to FX_i\to\dots\to FX_i\hookrightarrow X_i\to\dots\to X_i\twoheadrightarrow
	M_i^\bullet\] with $X_i^\bullet\in\E$. Thus we have $\oplus X_i^\bullet\twoheadrightarrow
	M^\bullet$ and we are done. 
	\end{proof}

\begin{lemma}
	$\Liota$ is exact on $\E$.
\end{lemma}
\begin{proof}
	Let $0\to (X,f)\to (Y,g)\to (Z,h)\to 0$ be an exact sequence in $\E$. For each
	$i=1,\dots,p-1$ we end up with the following commutative diagram where all rows and columns
	are exact
	\[\xymatrix{
		&0&0\\
		0\ar[r]& FZ_p\ar[u]\ar[r]&Z_i\ar[r]\ar[u]&\cok{h_1\dots h_i}\ar[r]&0 \\
		0\ar[r]& FY_p\ar[r]\ar[u]&Y_i\ar[u]\ar[r]&\cok{g_1\dots g_i}\ar[r]\ar[u]&0 \\ 
		0\ar[r]& FX_p\ar[r]\ar[u]&X_i\ar[r]\ar[u]&\cok{f_1\dots f_i}\ar[r]\ar[u]&0 \\
		&0\ar[u]&0\ar[u]
	}\] From the Snake Lemma we see that \[0\to \cok{f_1\dots f_i}
	\to \cok{g_1\dots g_i}\to \cok{h_1\dots h_i}\to 0\] is exact and so we are done.
\end{proof}

These two lemmas, together with Lemma~\ref{lem.exact_adjoints}, give us

\begin{proposition} \label{prop.Ext-iota}
Let $M \in \E$ and $N \in \Ares{\eta}[0^{\frac{1}{p-1}}]$. Then for any $n$ we have
\[ \Ext_{\Ares{\eta}[0^{\frac{1}{p-1}}]}^n( \Liota M, N ) = \Ext_{\A[\eta^{\frac{1}{p}}]}^n(M, \iota N). \]
In particular, for $M, N \in \Ares{\eta}[0^{\frac{1}{p-1}}]$ one obtains
\[ \Ext_{\A[\eta^{\frac{1}{p}}]}^n(\iota M, \iota N) = \Ext_{\Ares{\eta}[0^{\frac{1}{p-1}}]}^n( M, N ). \]
\end{proposition}

\begin{proof}
By the two lemmas above, we know that $\Liota$ and $\iota$ form a pair of exact adjoint functors between the exact categories $\E$ and $\Ares{\eta}[0^{\frac{1}{p-1}}]$. Thus it follows from Lemma~\ref{lem.exact_adjoints} that
\[ \Ext_{\Ares{\eta}[0^{\frac{1}{p-1}}]}^n( \Liota M, N ) = \Ext_{\E}^n(M, \iota N). \]
Now, since $\E$ is resolving in $\A[\eta^{\frac{1}{p}}]$, we have
\[ \Ext_{\E}^n(X, Y) = \Ext_{\A[\eta^{\frac{1}{p}}]}^n(X, Y) \]
for $X, Y \in \E$. In particular
\[ \Ext_{\Ares{\eta}[0^{\frac{1}{p-1}}]}^n( \Liota M, N ) = \Ext_{\A[\eta^{\frac{1}{p}}]}^n(M, \iota N). \]

The ``In particular" part now follows, since
\[ \Ext_{\A[\eta^{\frac{1}{p}}]}^n(\iota M, \iota N) = \Ext_{\Ares{\eta}[0^{\frac{1}{p-1}}]}^n( \underbrace{\Liota \iota M}_{= M}, N ). \]
\end{proof}

\begin{proposition}\label{prop:tilting}
Suppose $T$ is a cotilting object in
$\Ares\eta[0^{\frac{1}{p-1}}]$ and $U$ is cotilting in $\A$.
Then $E=\iota T\oplus \Rpi U$ is cotilting in $\A[\eta^{\frac{1}{p}}]$  if and only if $\Ext^i({\Rpi}U,\iota T)=0$ for all $i> 0$.
\end{proposition}
\begin{proof}
	 Cogeneration: Suppose $\Ext^i(M,E)=0$ for all  $i\geq 0$. Then
		$0=\Ext^i(M,\Rpi U)=\Ext^i(\pi M,U)$ and since $U$ is cogenerating this implies $\pi M=0$. In this case
		$M\simeq \iota N$ for some $N$. But then $0=\Ext^i(\iota N, \iota T)=\Ext^i({\Liota}\iota N, T)=\Ext^i(N, T)$ implies,
		since $T$ is cogenerating, that $N=0$, and so $M=0$.
 
		Rigidity: For $i>0$ we have $\Ext^i(\iota T, \iota T)=\Ext^i({\Liota}\iota T,T)=\Ext^i(T,T)=0$. 
	Also $\Ext^i(\Rpi U,{\Rpi}U)=\Ext^i(\pi \Rpi U,U)=\Ext^i(U,U)=0$
	and $\Ext^i(\iota T, \Rpi U)=\Ext^i(\pi \iota T, U)=\Ext^i(0, U)=0$. Finally, by assumption we have $\Ext^i(\Rpi U, \iota T)=0$ and so we
	are done.
\end{proof}

We now analyse this condition further. We define, an exact functor
\begin{align*}
	\Delta\colon \Ares{\eta}&\longrightarrow \Ares\eta[0^{\frac{1}{p-1}}]\\
	M&\longmapsto (0\to M\to M\to\dots\to M)
\end{align*}
which has a left and right adjoints
\begin{align*}
\LDelta(0\to M_1\to\dots\to
M_{p-1})&=M_{p-1}\\
	\RDelta(0\to M_1\to\dots\to
M_{p-1})&=M_1
\end{align*}

\begin{proposition}\label{prop:condition}
	Let $N\in\Ares\eta[0^{\frac{1}{p-1}}]$ and $M\in \A$. Then $\Ext^i(\Rpi M,\iota N)=0$ for all $i>0$ if:
	\begin{itemize}
		\item $\Ext^{i}(\Delta\ker\eta(M),N)=0$
			for all $i\ge 0$ and
		\item $\Ext^i(\Delta\cok\eta(M),N)=0$ for all $i>0$.
	\end{itemize}
\end{proposition}
\begin{proof}
	
We have the following exact sequence
\[0\to\iota\Delta\ker\eta(M)\to {\Lpi}M\to \Rpi M\to \iota \Delta \cok \eta(M)\to 0\]
which we break up as follows:
	\[\begin{array}{c}
		0\to\iota\Delta\ker\eta(M)\to {\Lpi}M\to C\to 0\\
		0\to C\to \Rpi M\to \iota \Delta \cok \eta(M)\to 0
\end{array}\]
Since for all $i\geq 0$ \[\Ext^i({\Lpi}M,\iota N)=
\Ext^i(M, \pi \iota N)=\Ext^i(M,0)=0\]
the first sequence implies that
\[ \Ext^i(C, \iota N) = \Ext^{i-1}( \iota \Delta \ker \eta(M), \iota N)
\overset{\text{Prop~\ref{prop.Ext-iota}}}{=} \Ext^{i-1}(\Delta \ker \eta(M), N). \]
Inserting this in the long exact sequence obtained from the second short exact sequence above, we obtain
\[\dots\to \underbrace{\Ext^i(\iota \Delta\cok \eta(M),\iota N)}_{= \Ext^i(\Delta\cok \eta(M), N)} \to\Ext^i(\Rpi
	M,\iota N)\to \underbrace{\Ext^i(C,\iota N)}_{\mathclap{= \Ext^{i-1}(\Delta \ker \eta(M), N)}} \to \dots, \]

from which the proposition follows.
\end{proof}

\begin{theorem}\label{thm:main}
	Suppose $T$ is cotilting in $\Ares\eta$ and $U$ is cotilting in $\A$. If
	\begin{enumerate}
		\item  $\eta(U)$ is injective, and
		\item $\Ext^i_{\Ares\eta}(\cok\eta(U),T)=0$ for all $i>0$.
	\end{enumerate}Then $\iota \delta(T)\oplus \Rpi U$ is cotilting $\A[\eta ^{\frac{1}{p}}]$.
\end{theorem}
\begin{proof}
	From (\ref{prop:tilting}) we require $\Ext^i(\Rpi U,\iota T)=0$. Now apply (\ref{prop:condition}) with $M=U$
	and $N=\delta(T)$ from (\ref{prop:tiltingobject}). The theorem then follows from the fact that, for all $i\geq 0$ and $M\in \Ares\eta$:
	\[\Ext^i(\Delta M, \delta(T))=\Ext^i(M,{\RDelta}\delta( T))=\Ext^i_{\Ares\eta}(M, T).\qedhere\]
\end{proof}

\section{Cotilting in the general case}\label{sec:general}
In this section we turn our attention to the more general category \[\AAA.\] We will give a criterion for this category to have a cotilting
object.

For each
$I\subseteq \{1,\dots,n\}$, we define the following full subcategory of $\A$:
\[\Ares{I}:=\left\{ M\in\A\mid \eta_i(M)=0\; \text{for all} \; i\in I\right\}.\]
Furthermore, assume that each such $\Ares{I}$ has a cotilting object $T_I$. In 
particular $T_{\emptyset}$ is a cotilting object in $\A$.

Before we proceed, we need to introduce several new categories, just
as we did in Section \ref{1weight}
whose cotilting objects will be used to construct the cotilting object we are seeking.

For $H,I,J\subseteq \{1,\dots,n\}$ with $H=\{a_1\dots a_m\}\subseteq I$  and 
$J=\left\{ b_1,\dots, b_\ell \right\}$ with $J\cap I=\emptyset$ let
\[\Ares{I}[\eta^J,0^{H}]:=\Ares{I}[\eta_{b_1}^{\frac{1}{p_{b_1}}},
	\dots,\eta_{b_\ell}^{\frac{1}{p_{b_\ell}}},0^{\frac{1}{p_{a_1}-1}},\dots,
0^{\frac{1}{p_{a_m}-1}}].\]

We have, for any $K=\{c_1,\dots,c_i\}\subseteq\{1,\dots,n\}$ satisfying $K\cap(I\cup J)=\emptyset$ a restriction functor 
\begin{align*}
	\mid_K&\colon\Ares{I}[\eta^J,0^{H}]\to \Ares{I\cup K}[\eta^J,0^{H}]\\
	M|_K&:=\cok\eta_{c_1}(\cok\eta_{c_2}(\dots\cok\eta_{c_i}(M)))
\end{align*} which is well defined since the $F_i$ commute.

If either $J$ or $H$ are empty, we leave them
out from the notation.
The category $\Ares{I}[\eta^\emptyset, 0^H]=\Ares{I}[0^H]$ has a special tilting object $T_I^H$
constructed iteratively from $T_I$ as we saw in Section \ref{1weight}. Note that $T_I^{\emptyset}=T_I$.

Let $H,I,J$ be as above, and $a,b\in H$. We have the following diagram where every row and column is a recollement:
\[\xymatrix{
	\Ares{I\setminus \{b\}}[ \eta^J,0^{H\setminus \{b\}} ]\ar[r]^{\iota^a}&
	\Ares{I\setminus \{a,b\}}[ \eta^{J\cup \{a\}},0^{H\setminus\{a,b\}} ] \ar[r]^{\pi^a}&
	\Ares{I\setminus \{a,b\}}[ \eta^J,0^{H\setminus \{a,b\}}]\\
	\Ares{I\setminus \{b\}}[ \eta^{J\cup \{b\}},0^{H\setminus\{b\}} ]\ar[u]^{\pi^b}\ar[r]^{\iota^a}&
	\Ares{I\setminus \{a,b\}}[ \eta^{J\cup \{a,b\}},0^{H\setminus\{a,b\}} \ar[r]^{\pi^a}]\ar[u]^{\pi^b}&
	\Ares{I\setminus \{a,b\}}[ \eta^{J\cup \{b\}},0^{H\setminus \{a, b\}}] \ar[u]^{\pi^b}\\
	\Ares{I}[\eta^J,0^{H}]\ar[r]^{\iota^a}\ar[u]^{\iota^b}&
	\Ares{I\setminus \{a\}}[ \eta^{J\cup \{a\}},0^{H\setminus\{a\}} ] \ar[r]^{\pi^a}\ar[u]^{\iota^b}&
	\Ares{I\setminus\{a\}}[ \eta^J,0^{H\setminus \{a\}}]\ar[u]^{\iota^b}
}\]
\begin{remark}
Note that we have abused notation slightly by calling many different functors $\iota^a$. However, no confusion should arise for 
they all have different domains and codomains and the correct one is hence clear from context. The same applies to $\pi^a$ and
$\Rpi^a$ as well.
\end{remark}

%\old{
%\begin{lemma}\label{commute} The functors $\iota$ and $\Rpi$ commute. More precisely,
%	the following $3$ diagrams commute:
%	\begin{enumerate}
%	\item 
%	\[
%	\xymatrix{
%		\Ares{I\setminus \{b\}}[ \eta^J,0^{H\setminus \{b\}}]\ar[r]^{\iota^a}\ar[d]_{{\Rpi^b}}&
%		\Ares{I\setminus \{a,b\}}[
%		\eta^{J\cup \{a\}},0^{H\setminus\{a,b\}} ]\ar[d]^{
%			{\Rpi^b}}\\
%	\Ares{I\setminus \{b\}}[ \eta^{J\cup \{b\}},0^{H\setminus\{b\}} ]\ar[r]^{\iota^a}&
%	\Ares{I\setminus \{a,b\}}[ \eta^{J\cup \{a,b\}},0^{H\setminus\{a,b\}}]}
%	\]
%	\item
%	\[
%	\xymatrix{
%	\Ares{I\setminus \{a,b\}}[
%	\eta^{J\cup \{a\}},0^{H\setminus\{a,b\}} ] \ar[d]_{
%		{\Rpi^b}}&
%	\Ares{I\setminus \{a,b\}}[ \eta^J,0^{H\setminus
%	\{a,b\}}]\ar[l]_{ {\Rpi^a}}\ar[d]^{{\Rpi^b}}\\
%	\Ares{I\setminus \{a,b\}}[ \eta^{J\cup \{a,b\}},0^{H\setminus\{a,b\}}]&
%	\Ares{I\setminus \{a,b\}}[ \eta^{J\cup
%	\{b\}},0^{H\setminus\{a,b\}}] \ar[l]^{ {\Rpi^a}}\\
%	}\]
%	\item
%	\[\xymatrix{
%	\Ares{I}[\eta^J,0^{H}]\ar[r]^{\iota^a}&
%	\Ares{I\setminus \{a\}}[ \eta^{J\cup \{a\}},0^{H\setminus\{a\}} ]\\
%	\Ares{I}[\eta^J,0^{H}]\ar[r]^{\iota^a}\ar[u]^{\iota^b}&
%	\Ares{I\setminus \{a\}}[ \eta^{J\cup \{a\}},0^{H\setminus\{a\}}]\ar[u]_{\iota^b}
%	}\]
%	\end{enumerate} 
%	%A similar result holds if we replace $\Rpi$ with $\Lpi$.
%\end{lemma}
%}

	\begin{lemma} \label{commute}
In the diagram of recollements above, all the squares (including original functors and adjoint functors) commute, except $\Riota$ and $\Liota$. In particular we have the following three equalities which we will use later:
\begin{enumerate}
\item $\iota^a \iota^b = \iota^b \iota^a$ (i.e.\ the left lower square of the diagram commutes);
\item $\iota^a \Rpi^b = \Rpi^b \iota^a$;
\item $\Rpi^a \Rpi^b = \Rpi^b \Rpi^a$. 
\end{enumerate} 
\end{lemma}

\begin{proof}
	This is a simple, straight forward calculation.
\end{proof}
In light of this lemma we define, for $H=\left\{ a_1,\dots,a_h \right\}\subseteq \{1,\dots,n\}$ and an object $M$ in an appropriate category
\[\iota^HM:=\iota^{a_1}\circ\dots\circ \iota^{a_h}(M).\] Similarly we define $\Rpi^H$ and $\pi^H$.

Similarly to the case with only one weight, we need to control how the adjoint pair
$(\iota^{a}_\lambda, \iota^a)$ behave with respect to $\Ext$. We therefore need a more general version of the assumption used earlier: 

\begin{assumption}\label{generalassumption}
From now on, assume that  for	all $I\subseteq \{1,\dots,n\}$ and $i\in \{1,\dots,n\}\setminus I$ the category $\A_I$ has enough objects $M$ such that $\eta_i(M)$ is a monomorphism.

\end{assumption}
\begin{lemma}
	Let $H,I,J\subseteq \{1,\dots,n\}$ with $H\subseteq I$ and $J\cap I =\emptyset$ and
	$a\in H$. Suppose, $M,N\in \Ares{I}[\eta^J,0^{H}]$. We have, for all $i\ge0$ :
	\[\Ext^i(\iota^aM, \iota^a N)=\Ext^i(M,N)\]
\end{lemma}

\begin{proof}
By Assumption~\ref{generalassumption} we have that $\Ares{I\setminus\{a\}}$ has enough objects on which $\eta_a$ is monomorphism. Similarly to the proof of \ref{lemma:resolving} one sees that this implies that also $\Ares{I\setminus\{a\}}[ \eta^J,0^{H\setminus \{a\}}]$ has enough objects such that $\eta_a$ is monomorphism.
The result then follows from \ref{prop.Ext-iota} and the observation that 

 \[\Ares{I}[\eta^J,0^{H}] = \left( \Ares{I\setminus\{a\}}[ \eta^J,0^{H\setminus
 \{a\}}]\right)_{\eta_a}[0^{\frac{1}{p_a-1}}]. \]  
\end{proof}

\begin{lemma}Let $I\subseteq \{1,\dots,n\}$ and $J\subseteq I$.\label{lemma:computation}
	\[\Ext^i(\Rpi^JM,\iota^J T_I^I)=0\] for all $i>0$ if:
	\begin{itemize}
		\item $\eta_{a}(M|_{J'})$ is injective for all $J'\subset J$ and $a\in J\setminus J'$
		\item $\Ext^i(M|_J, T_I^{I\setminus J})=0$
	\end{itemize}
\end{lemma}
\begin{proof} We have, for all $a\in J$
	\[\Ext^i(\Rpi^JM,\iota^J T_I^I)=\Ext^i(\Rpi^a(\Rpi^{J\setminus\{a\}}M),\iota_{a}(\iota^{J\setminus\{a\}} T_I^I))\]
	Hence, using (\ref{prop:condition}) we see that $\Ext^i(\Rpi^J M,\iota^JT_I^I)=0$ for all $i>0$ if
	\begin{itemize}
		\item $\Ext^i(\Delta\ker\eta_{a}(\Rpi^{J\setminus\{a\}}M),
			\iota^{J\setminus\{a\}}T_I^I)=0$ for $i\ge 0$ and
		\item $\Ext^i(\Delta\cok\eta_{a}(\Rpi^{J\setminus\{a\}}M),
			\iota^{J\setminus\{a\}}T_I^I)=0$ for $i>0$.
	\end{itemize} Since 
	\[{\ker}\;\eta_{a}(\Rpi^{J\setminus\{a\}} M)=\Rpi^{J\setminus\{a\}}
	{\ker}\;\eta_{a}(M)\quad{\rm and}\quad({\RDelta})\iota^{J\setminus\{a\}}T_I^I=\iota^{J\setminus\{a\}}T_I^{I\setminus \{a\}}\]
	and so the two conditions become:
	\begin{itemize}
		\item $\Ext^i(\Rpi^{J\setminus\{a\}} \ker\eta_{a}M,
			\iota^{J\setminus\{a\}}T_I^{I\setminus \{a\}})=0$
			for $i\ge0$ and
		\item $\Ext^i(\Rpi^{J\setminus\{a\}} M|_{\{a\}},
			\iota^{J\setminus\{a\}}T_I^{I\setminus \{a\}})=0$
			for $i>0$.
	\end{itemize} Now repeat this procedure $|J|-1$ more times to get the result.
\end{proof}

\begin{theorem}\label{thm:general}Let $\A$ be an abelian category equipped with endofunctors $F_i$
	and natural transformation $\eta_i$ as in Section~\ref{sec:setup}, satisfying Assumption~\ref{generalassumption}.
	Assume there are cotilting objects $T_H$ in $\Ares{H}$, such that  for all $H\cap J =\emptyset$ and $a\not\in H\cup J$
	\begin{itemize}
		\item $\eta_a(T_H|_{J})$ is injective
		\item $\Ext^i_{\Ares{H\cup J}}(T_H|_J,
			T_{H\cup J})=0$ for all $i>0$.
	\end{itemize}
	Then, with the notations introduced above, 
	\[T:=\bigoplus_{H\subseteq \{1,\dots,n\}}\Rpi^{[1,n]\setminus H}\iota^{H} T_{H}^{H}\]
	is a cotilting object in $\AAA$.
\end{theorem}
\begin{proof} First we introduce the following notation: for $I\subseteq \{1,\dots,n\}$ let $\bar I:=[1,n]\setminus I$.

	Rigidity: We compute \[\Ext^i(\Rpi^{\bar H}\iota^H T_H^H, \Rpi^{\bar I}\iota^I T_I^I)\] for all $H,I\subseteq \{1,\dots,n\}$ and $i\ge 1$.
	%$a=n-b, \; c=n-d, \;b=|H|$ and $d=|I|$. If $bc\neq 0$, 
	If $H\cap \bar I\neq \emptyset$
	then using (\ref{commute}):
	\[\Ext^i(\Rpi^{\bar H}\iota^H T_H^H, \Rpi^{\bar I}\iota^I T_I^I)=\Ext^i(\pi^a\iota^a\iota^{H\setminus\{a\}}\Rpi^{\bar
	H}T_H^H,\Rpi^{\bar I\setminus\{a\}}\iota^IT_I^I)=0\]
	where $a\in H\cap \bar I$, since $\pi^a \iota^a=0$. 
	Thus we consider the case where $H\cap\bar I=\emptyset$ or, equivalently, $H\subseteq I$:
	
	If $I=H$, then since $\pi \Rpi=\Liota\iota=\id$, we have\[\Ext^i(\Rpi^{\bar H}\iota^H T_H^H,\Rpi^{\bar H}\iota^H T_H^H)=
	\Ext^i(T^H_H,T^H_H)=0\] since $T^H_{H}$ is cotilting.
	
	Finally, suppose $H\subset I$ and let $J=I\setminus H$. Then we have
	\[\Ext^i(\Rpi^{\bar H}\iota^H T_H^H, \Rpi^{\bar I}\iota^I T_I^I)=
	\Ext^i(\Rpi^JT_H^H, \iota^JT_I^I). \] By (\ref{lemma:computation}) this vanishes when
	\begin{itemize}
		\item $\eta_{a}(T^H_H|_{J'})$ is injective for all $J'\subset J$ and $a\in J\setminus J'$
		\item $\Ext^i(T^H_H|_J, T_I^{I\setminus J})=\Ext^i(T^H_H|_J, T_I^H)=0$ for all $i>0$.
	\end{itemize} Thus rigidity follows from the assumptions of the theorem and Lemma \ref{lemma:ext} applied $|H|$ times.

	Cogeneration: 
	Suppose $\Ext^i(M,T)=0$ for all $i\geq 0$. We aim to show $M=0$.
	We  do so by proving, that for all $I\subseteq \{1,\dots,n\}$, we have $\pi^IM=0$.
	We have\[0=\Ext^i(M,\Rpi^{[1,n]}T_{\emptyset})=\Ext^i(\pi^{[1,n]}M,T_{\emptyset})\]
	and so $\pi^{[1,n]} M=0$. We proceed by reverse induction on $|I|$.

	Suppose, that $\pi^{J}M=0$ for all $|J|\ge k+1$. Let $I\subseteq \{1,\dots,n\}$ such that $|I|=k$.
	Then for all $a\in \bar I$ we have $\pi^a\pi^IM=0$ so $\pi^IM=\iota^aN'$ for some $N'$. Hence
	 $\pi^IM=\iota^{\bar I}N$ for some $N$. 
	\[0=\Ext^i(M,\Rpi^{ I}\iota^{\bar I} T_{\bar I}^{\bar I})=\Ext^i(\pi^IM,\iota^{\bar I}T_{\bar I}^{\bar I})=\Ext^i(\iota^{\bar
	I}N,\iota^{\bar I}T_{\bar I}^{\bar I})=\Ext^i(N,T_{\bar I}^{\bar I})\]
	and so $N=0$ and hence $\pi^IM=0$ for all $\pi^I$ with $|I|=k$. Therefore, $\pi^IM=0$ for all $I\subseteq \{1,\dots,n\}$, in
	particular $\pi^\emptyset M:=M=0$.
	%Finally,
	%\[0=\Ext^{i}(M,\iota^{[1,n]}T_{[1,n]}^{[1,n]})=\Ext^i(\iota^{[1,n]}Z,\iota^{[1,n]}T_{[1,n]}^{[1,n]})=\Ext^i(Z,T_{[1,n]}^{[1,n]})\]
	%and so $Z=0$.
\end{proof}

\section{Global dimension}\label{sec:global}

In this section we study the global dimension of the categories $\A [\eta^{\frac{1}{p}}]$. The main
aim is Theorem~\ref{thm.gldim_main}, showing that under certain assumptions (the most important of which is that $F$ is an autoequivalence) the global dimension of $\A [\eta^{\frac{1}{p}}]$ equals that of $\A$.

We start by considering the categories on the left side of the recollement of Proposition~\ref{prop.recollement}. (The abelian category here is called $\A_{\eta}$ because these are the categories we want to apply this to. However for this lemma this is just an arbitrary abelian category.)

\begin{lemma} \label{gldim_zero}
Let $\A_{\eta}$ be abelian, and $p \geq 2$. Then
\begin{itemize}
\item $\gldim \A_{\eta} [0^{\frac{1}{p-1}}] \leq \gldim \A_{\eta} +1$ (and in fact we have equality unless $p = 2$);
\item if $M \in \A_{\eta} [0^{\frac{1}{p-1}}]$ such that all morphisms $f_2^M, \ldots, f_{p-1}^M$ are epi then $\injdim M \leq \gldim \A_{\eta}$.
\end{itemize}
\end{lemma}

We do not prove this lemma here, but assume it holds for a given $p$. Note that this is justified for $p=2$. (In that case $\A_{\eta} [0^{\frac{1}{p-1}}] = \A_{\eta}$.) In the further discussion in this section we will reach an argument showing that the lemma then also holds for $p+1$, thus proving it inductively. See Remark~\ref{rem.induct_0}

Now we start considering the general case. Throughout the following morphisms and resulting exact sequences will play a role.

\begin{observation} \label{obs.ker_coker_unit}
Let $\A$, $F$, and $\eta$ as in Proposition~\ref{prop.recollement}. Let $X \in \A[\eta^{\frac{1}{p}}]$.
\begin{enumerate}
\item
For the unit $\varepsilon_X \colon X \to \Rpi \pi X$ we have
\begin{align*}
\ker \varepsilon_X & = \iota \Riota X \qquad \text{and} \\
\cok \varepsilon_X & = \iota(0 \to \cok f_2^X \cdots f_p^X \to \cdots \to \cok f_p^X).
\end{align*}
We note that all the non-zero maps in $[0 \to \cok f_2^X \cdots f_p^X \to \cdots \to \cok f_p^X]$ are epimorphisms.
\item
For the counit $\varphi_X \colon \Lpi \pi X \to X$ we have
\begin{align*}
\ker \varphi_X & = \iota(0 \to \ker f_1^X \to \cdots \to \ker f_1^X \cdots f_{p-1}^X) \qquad \text{and} \\
\cok \varphi_X & = \iota \Liota X.
\end{align*}
\end{enumerate}
\end{observation}

We first study extensions in $\A[\eta^{\frac{1}{p}}]$ where the first term is in the image of the functor $\iota$.

\begin{lemma} \label{lem.ext_from_iota}
Let $\A$, $F$, and $\eta$ as in Section~\ref{1weight}. For $X \in \A_{\eta}[0^{\frac{1}{p-1}}]$ and $Y \in \A[\eta^{\frac{1}{p}}]$ we have
\[ \Ext^n_{\A[\eta^{\frac{1}{p}}]}( \iota X, Y) = 0 \quad \forall n > \gldim \A_{\eta} + 1. \]
If moreover all the maps $f_2^Y, \ldots, f_p^Y$ are epimorphisms, then the equality also holds for $n = \gldim \A_{\eta} + 1$.
\end{lemma}

\begin{proof}
We first observe that
\[ \Ext^n_{\A[\eta^{\frac{1}{p}}]}(\iota X, \Rpi \pi Y) = \Ext^n_{\A}( \underbrace{\pi \iota X}_{=0}, \pi Y) = 0 \; \forall n, \]
so the $\Ext$-space of the lemma vanishes provided
\begin{enumerate}
\item $\Ext^n_{\A[\eta^{\frac{1}{p}}]}(\iota X, \ker \varepsilon_Y) = 0$, and
\item $\Ext^{n-1}_{\A[\eta^{\frac{1}{p}}]}(\iota X, \cok \varepsilon_Y) = 0$.
\end{enumerate}
For the first space we use Observation~\ref{obs.ker_coker_unit} to simplify
\[ \Ext^n_{\A[\eta^{\frac{1}{p}}]}(\iota X, \ker \varepsilon_Y) = \Ext^n_{\A_{\eta}[0^{\frac{1}{p-1}}]}(X, \Riota Y), \]
so this space vanishes provided $n > \gldim \A_{\eta}[0^{\frac{1}{p-1}}]$, and hence by Lemma~\ref{gldim_zero} for $n > \gldim \A_{\eta} + 1$. Moreover, using the second part of Lemma~\ref{gldim_zero}, we see that this bound may be improved by $1$ provided all the maps $f_2^{\Riota Y}, \ldots f_{p-1}^{\Riota Y}$ are epimorphisms. This holds provided the corresponding maps $f_2^Y, \ldots, f_{p-1}^Y$ are epi.

For the second space we use the remark in the first point of Observation~\ref{obs.ker_coker_unit}. Note that this precisely tells us that we are in the situation of the second point of Lemma~\ref{gldim_zero}, whence
\[ \Ext^{n-1}_{\A[\eta^{\frac{1}{p}}]}(\iota X, \cok \varepsilon_Y) = 0 \quad \forall n-1 > \gldim \A_{\eta}. \]
Finally we note that if all the maps $f_2^Y, \ldots, f_p^Y$ are epi then $\cok \varepsilon_Y = 0$, so the space in the second point vanishes. 
\end{proof}

In the next step we assume that the first object lies in the set $\E$, that is that the map $f_1^X \cdots f_{p-1}^X$ is a monomorphism.

\begin{lemma} \label{lem.Ext_from_E}
Let $X \in \E$ and $Y \in \A[\eta^{\frac{1}{p}}]$. Then
\[ \Ext^n_{\A[\eta^{\frac{1}{p}}]}(X, Y) = 0 \quad \forall n > \max \{ \gldim \A, \; \gldim \A_{\eta} + 1 \}. \]
If moreover all of the maps $f_2^Y, \ldots, f_p^Y$ are epimorphisms, then
\[ \Ext^n_{\A[\eta^{\frac{1}{p}}]}(X, Y) = 0 \quad \forall n > \max \{ \gldim \A,\; \gldim \A_{\eta} \}. \]
\end{lemma}

\begin{proof}
We start by observing that $X \in \E$ is equivalent to $\ker \varphi_X = 0$, whence we have the short exact sequence
\[ 0 \to \Lpi \pi X \to X \to \iota \Liota X \to 0. \]
Therefore it suffices to consider the two $\Ext$-spaces $\Ext^n_{\A[\eta^{\frac{1}{p}}]}(\Lpi \pi X, Y)$ and
$\Ext^n_{\A[\eta^{\frac{1}{p}}]}(\iota \Liota X, Y)$.

For the first of these we have
\[ \Ext^n_{\A[\eta^{\frac{1}{p}}]}(\Lpi \pi X, Y) = \Ext^n_{\A}(\pi X, \pi Y), \]
so this vanishes for $n > \gldim \A$.

For the second one we use Lemma~\ref{lem.ext_from_iota} above.
\end{proof}

\begin{remark} \label{rem.induct_0}
We observe that we have now completed an inductive proof of the upper bounds in Lemma~\ref{gldim_zero}. In fact, in the case $F = 0$ we have $\E = \A[0^{\frac{1}{p}}]$, so there is no restriction on $X$ in the lemma above. The equality claimed in parantesis in Lemma~\ref{gldim_zero} follows from the following result.
\end{remark}

\begin{lemma} \label{lem.lower_bound}
Let $\A$, $F$, and $\eta$ as in Section~\ref{1weight}, and $p \geq 2$. Then
\[ \gldim \A[\eta^{\frac{1}{p}}] \geq \max \{ \gldim \A,\; \gldim \A_{\eta} + 1 \}. \]
\end{lemma}

\begin{proof}
We have $\gldim \A[\eta^{\frac{1}{p}}] \geq \gldim \A$ because, for $X, Y \in \A$ we have
\[ \Ext^n_{\A[\eta^{\frac{1}{p}}]}(\Lpi X, \Lpi Y) = \Ext^n_{\A}(X, \pi \Lpi Y) = \Ext^n_{\A}(X, Y). \]

For the second part recall that for $X,Y \in \A_{\eta}$ we set
\[ \Delta Y = [0 \to Y \to Y \to \cdots \to Y] \in \A_{\eta}[0^{\frac{1}{p-1}}]. \]
Now we observe that we have an epimorphism
\[ p \colon  \Rpi Y \to \iota \Delta Y \]
in the category $ \A[\eta^{\frac{1}{p}}]$. We may explicitly describe the kernel of $p$ as
\[ \ker p = [FY \to 0 \to \cdots \to 0 \to Y]. \]
We have the following exact sequence \[0\to \ker p\to \Rpi Y\to \iota\Delta Y\to 0\]
and since $\Ext^n_{\A[\eta^{\frac{1}{p}}]}( \iota \Delta X, \Rpi Y) = \Ext^n_{\A}(\pi \iota \Delta X, Y) = 0$ 
for all $n$ we obtain
\[ \Ext^n_{\A[\eta^{\frac{1}{p}}]}(\iota \Delta X, \ker p) =
\Ext^{n-1}_{\A[\eta^{\frac{1}{p}}]}(\iota \Delta X, \iota \Delta Y) = \Ext^{n-1}_{\A_{\eta}}(X, Y). \]
It follows that
\[ \gldim \A[\eta^{\frac{1}{p}}] \geq \gldim \A_{\eta} + 1. \]
\end{proof}

We are now ready to prove the main result of this section, giving the precise value of the global dimension of the category $\A[\eta^{\frac{1}{p}}]$ under the assumption that $F$ is an equivalence.

\begin{theorem} \label{thm.gldim_main}
Let $\A$, $F$, and $\eta$ as in Section~\ref{1weight}, and assume additionally that $F$ is an equivalence. Then
\[ \gldim \A[\eta^{\frac{1}{p}}] = \max \{ \gldim \A, \gldim \A_{\eta} + 1 \}. \]
\end{theorem}

One key ingredient for the proof is the following observation.

\begin{observation}
The functor $F$ induces an endofunctor $F_{\frac{1}{p}}$ of $\A[\eta^{\frac{1}{p}}]$, given by
\[ F_{\frac{1}{p}} ( FX_p \to X_1 \to \cdots \to X_p) = [ FX_{p-1} \to FX_p \to X_1 \to \cdots \to X_{p-1} ]. \]
Moreover, if $F$ is an autoequivalence of $\A$, then $F_{\frac{1}{p}}$ is an autoequivalence of $\A[\eta^{\frac{1}{p}}]$.
\end{observation}

\begin{proof}
For $X \in \A[\eta^{\frac{1}{p}}]$, we consider the short exact sequence
\[ 0 \to \ker \varepsilon_X \to X \to \operatorname{im} \varepsilon_X \to 0. \]
We may explicitly describe the rightmost term by
\[ \operatorname{im} \varepsilon_X = [F X_p \to \operatorname{im} f_2^X \cdots f_p^X \to \operatorname{im} f_3^X \cdots f_p^X \to \cdots \to X_p], \]
and in particular all the maps $f_2^{\operatorname{im} \varepsilon_X}, \ldots f_p^{\operatorname{im} \varepsilon_X}$ are monomorphisms. It follows that
\[ F_{- \frac{1}{p}}(  \operatorname{im} \varepsilon_X ) \in \E. \]
It follows that for all $Y\in \A[\eta^{\frac{1}{p}}]$:
\[ \Ext^n_{\A[\eta^{\frac{1}{p}}]}(\operatorname{im} \varepsilon_X, Y) = \Ext^n_{\A[\eta^{\frac{1}{p}}]}( F_{- \frac{1}{p}}( \operatorname{im} \varepsilon_X), F^{- \frac{1}{p}}(Y)) = 0 \]
for $n >   \max \{ \gldim \A, \gldim \A_{\eta} + 1 \}$. (The first equality holds since $F_{\frac{1}{p}}$ is an autoequivalence, the second is Lemma~\ref{lem.Ext_from_E}.

On the other hand, since $\ker \varepsilon_x=\iota\Riota X$, we also have that
\[ \Ext^n_{\A[\eta^{\frac{1}{p}}]}(\ker \varepsilon_X, Y) = 0 \quad \forall n > \gldim \A_{\eta} + 1 \]
by Lemma~\ref{lem.ext_from_iota}.

Now, since $X$ is the middle term of a short exact sequence with $\ker \varepsilon_X$ and $\operatorname{im} \varepsilon_X$ as end terms, we have
\[ \Ext^n_{\A[\eta^{\frac{1}{p}}]}(X, Y) = 0 \quad \forall n > \max \{ \gldim \A,\; \gldim \A_{\eta} + 1 \}. \]
That is we have the upper bound for the global dimension of $A[\eta^{\frac{1}{p}}]$. The fact that this is also a lower bound was seen in Lemma~\ref{lem.lower_bound} above.
\end{proof}

\begin{remark}
In case that $F$ is not an equivalence, one may extend the argument in the proof of Lemma~\ref{lem.Ext_from_E} to the case where $X$ is not necessarily in $\E$: In that case one has to account for a possible kernel of $\varphi_X$, resulting in the weaker upper bound
\[ \gldim \A[\eta^{\frac{1}{p}}] \leq \max \{ \gldim \A, \;\gldim \A_{\eta} + 2 \}. \]
However we do not have any examples of Theorem~\ref{thm.gldim_main} failing when $F$ is not an equivalence.
\end{remark}

Finally, we may apply Theorem~\ref{thm.gldim_main} repeatedly to obtain the global dimension of categories of the form $\A[\eta_1^{\frac{1}{p_1}}, \ldots, \eta_n^{\frac{1}{p_n}}]$.

\begin{corollary}
	In the general situation of Theorem \ref{thm:general}, and assuming further that all the
	$F_i$ are autoequivalences,  we have
\[ \gldim \A[\eta_1^{\frac{1}{p_1}}, \ldots, \eta_n^{\frac{1}{p_n}}] = \max_{I \subseteq \{1, \ldots, n\}} \gldim \A_I + |I|. \]
(Here we set $\gldim 0 = - \infty$, or alternatively let the maximum run over all $I$ such that $\A_I \neq 0$.)
\end{corollary}

	\begin{proof}
		We can construct the category $\AAA$ iteratively using the fact that
		\[\A[\eta_1^{\frac{1}{p_1}},\dots,\eta_i^{\frac{1}{p_i}},\eta_{i+1}^{\frac{1}{p_{i+1}}}]=
		\A[\eta_1^{\frac{1}{p_1}},\dots,\eta_i^{\frac{1}{p_i}}][\eta_{i+1}^{\frac{1}{p_{i+1}}}]
	\] where we have extended the action of $F_{i+1}$ to
	$\A[\eta_1^{\frac{1}{p_1}},\dots,\eta_i^{\frac{1}{p_i}}]$ component wise. Since the
	$F_i$ commute, this construction is well defined and equals our original construction. Thus
	the result follows from Theorem \ref{thm:general} applied repeatedly.
\end{proof}

\section{Applications to orders on projective varieties}\label{sec:order}
In \cite{IL} \emph{Geigle-Lenzing (GL) orders} on $\P^d$ were used to study Geigle-Lenzing weighted projective spaces which
in turn were introduced in \cite{HIMO}. We have already introduced GL order in
Definition~\ref{def:order}:
they are orders made up of tensor products of sheaves of algebras of the form
\[T_{p_i}(\O,\O(-L_i)) := 
\underbrace{\begin{bmatrix}
			\O&\O(-L_i)&\dots&\O(-L_i)&\O(-L_i)\\
			\O&\O&\dots&\O(-L_i)&\O(-L_i)\\
			\vdots&\vdots&\ddots&\vdots&\vdots\\
			\O&\O&\dots&\O&\O(-L_i)\\
			\O&\O&\dots&\O&\O
	\end{bmatrix}}_{p_i}\]

The connection that the category $\AAA$ has to orders is described in the following proposition.
\begin{proposition}  Let $X$ be a projective variety over $k$
	and $L_1,\dots, L_n$ effective Cartier divisors on $X$.
Let $\A=\coh X$,  $F_i:=-\otimes_X \O(-L_i)$ and 
$\eta_i(\F)\colon F_i\F\hookrightarrow \F$ be the natural inclusions. If 
\[
	A=\bigotimes_{i=1}^nA_i, \quad A_i=T_{p_i}(\O,\O(-L_i)) 
\]
then\[\mod A\simeq\A[\eta_1^{\frac{1}{p_1}},\dots,\eta_n^{\frac{1}{p_n}}]\]
\end{proposition}
\begin{proof}
	Let $e_{i}\in H^0(X,A_i)$ be the global section with a $1$ in the $(i,i)$-entry
	and $0$ elsewhere.
	
	We define the $\Phi\colon \mod A\to \A[\eta_1^{\frac{1}{p_1}},\dots,\eta_n^{\frac{1}{p_n}}]$ as follows: for $M\in \mod A$
	and $\a\in S=\{1,\dots,p_1\}\times\dots\times \{1,\dots,p_n\}$ set \[M_\a:=(e_{a_1}\otimes\dots\otimes e_{a_n})M.\] 
	As before we extend this to allow the components of $\a$ to be
		$0$ and treat them as functors $F_i$ i.e.\ $M_{\a} := F_i M_{\a + p_i \e_i}$ if $a_i=0$.
	
	Now, set
	\[f^i_\a\colon M_{\a - \e_i} \to M_{\a}\] to be the natural map coming from the $A_i$-module
	structure.
\end{proof}
From now on, we assume the following:
\begin{assumption}\label{general position}
$X$ is smooth 	and $D=\sum L_i$ is a simple
		normal crossing divisor i.e. for all $x\in \rm{Supp}\;D$ the local equations
		$f_i$ of $L_i$ form a regular sequence in $\O_{X,x}$.
\end{assumption}

\begin{proposition}
$A$ has finite global dimension.	
\end{proposition}

\begin{proof}
	The case $X=\P^d$ was proved in \cite[Proposition 2.13]{IL}. The proof remains unchanged for an arbitrary smooth $X$ and $L_i$.
\end{proof}
\begin{corollary}[\cite{IL}, Proposition 5.2]
	Let $T$ be a tilting object in $\mod A$. 
	\begin{itemize}
		\item $\DDD^{\bo}(\mod A)=\thick T$.
		\item There is a triangle equivalence $\DDD^{\bo}(\mod A)\simeq\DDD^{\bo}(\mod\End_A(T))$.
	\end{itemize}
\end{corollary}
\begin{proposition}[Serre duality]\label{prop:serre}
	Let $A$ be a GL order as before and suppose ${\rm dim }\;X=d$. Let \[\omega_A:={\mathcal Hom}_X(A, \omega_X)\]which is an invertible $A$-bimodule. Then for any
	$M,N\in\mod A$ we have \[\Ext^i_A(M,N)=D\Ext^{d-i}_A(N,\omega_A\otimes_A M)\] where $D(-)=\Hom_k(-,k)$.
\end{proposition}
\begin{proof}
	This proof is adapted from \cite{AdJ}. Let $h^i\colon\DDD^b(\mod A)\to \mod A$ be the $i$-th cohomology functor, 
	$\RGamma$ the right derived functor of the global sections functor $\Gamma$, $\RHOM_A(-,N)$ the right derived functor of the
	sheaf hom functor $\mathcal Hom_A(-, N)$ and $-\otimes^\L_A N$ the left derived functor of the tensor functor. 
	For simplicity, we introduce some more notation: let $\H^i:=h^i\circ \RGamma$ be the
	hypercohomology functor and for $M\in\DDD(\mod A)$ let $M^*:=\RHOM_A(M,A)$. With this we have:
	\begin{align*}
		\Ext^i_A(M,N)&=\H^i(\RHOM_A(M,N))=\H^i(M^*\otimes^\L_AN)\simeq D\H^{d-i}(\RHOM_\O(M^*\otimes_A^\L N,\O)\otimes_\O^\L\omega_X)\\
		&=D\H^{d-i}(\RHOM_\O(M^*\otimes_A^\L N,\omega_X))=D\H^{d-i}(\RHOM_A(N, \RHOM_\O(M^*,\omega_X)))\\
		&=D\H^{d-i}(\RHOM_A(N,\omega_A\otimes_A^\L M))\\
		&=D\Ext^{d-i}_A(N,\omega_A\otimes_A M)
	\end{align*}and we are done.
\end{proof}
\begin{corollary}
	$T\in\mod A$ is tilting if and only if it cotilting.
\end{corollary}
\begin{proof}
	Follows immediately from \ref{prop:serre}.
\end{proof}

We now translate our results from Sections \ref{1weight} and \ref{sec:general} to the category $\mod A$ but in light of the previous result we
will say $T$ is tilting, as opposed to cotilting.

Firstly, for each $i=1,\dots,n$ we have the following recollement, which is a sheafified
	version of the standard recollement we presented in Example \ref{eg:standard}.
Let $e_i$ be the global idempotent of $A_i$ with $1$ in the bottom right entry and $0$'s elsewhere. We
have \[\mod \frac{A}{\langle e_i\rangle}\xrightarrow{\iota}A\xrightarrow{\pi}e_iAe_i\] which is written
out in full looks like:
\[\mod
\underbrace{\begin{bmatrix}
	\O_{L_i}&0&\dots&\dots&0\\
	\O_{L_i}&\O_{L_i}&0&\dots&0\\
	\vdots&\vdots&&&\vdots\\
	\O_{L_i}&\O_{L_i}&\dots&\dots&\O_{L_i}
\end{bmatrix}}_{p_i-1}\xrightarrow{\iota}
\mod\underbrace{\begin{bmatrix}
	\O&\O(-L_i)&\dots&\dots&\O(-L_i)\\
	\O&\O&\O(-L_i)&\dots&\O(-L_i)\\
	\vdots&\vdots&&&\vdots\\
	\O&\O&\dots&\dots&\O
\end{bmatrix}}_{p_i}\xrightarrow{\pi}\coh X
\]
with 
\begin{align*}
	\iota(M)&=
		\begin{bmatrix}
			\O\\
			\vdots\\ \O\\0
		\end{bmatrix}\otimes_X M,\quad
	\pi(N)=e_iN=
	\begin{bmatrix}
		0&\dots&0\\
		\vdots&&\vdots\\
		0&\dots&1
	\end{bmatrix}N\\
	\Rpi(\mathcal F)&=\Hom_X(e_iA,\mathcal F)=Af_i\otimes_X\mathcal F=
	\begin{bmatrix}
		\O\\\vdots\\\O
	\end{bmatrix}\otimes_X\mathcal F
\end{align*}
where $f_i$ the global idempotent with $1$ in the top left entry and $0$'s elsewhere.

Furthermore, if $T$ is a tilting object in $\coh L_i$ then the
tilting object in 
\[\mod B:= \mod
\underbrace{\begin{bmatrix}
	\O_{L_i}&0&\dots&\dots&0\\
	\O_{L_i}&\O_{L_i}&0&\dots&0\\
	\vdots&\vdots&&&\vdots\\
	\O_{L_i}&\O_{L_i}&\dots&\dots&\O_{L_i}
\end{bmatrix}}_{p_i-1}\] is $B\otimes_{L_i} T$ which as an $\O_{L_i}$-module
is just
\[
	\begin{bmatrix}
		\O\\
		\O^2\\
		\vdots\\ \O^{p_i-1}
	\end{bmatrix}\otimes_{L_i} T
\]
Thus, applying (\ref{thm:general}) to this setup, and using the notation from Section \ref{sec:introduction} we see that we have proven:
\begin{theorem}\label{Thm:order}Let $\A=\mod A$ be as
	above.
	If for all $I,J\subseteq \{1,\dots,n\}$ with $I\cap
	J=\emptyset$ and $j\not\in I\cup J$
	\begin{enumerate}
		\item \label{cond1} $T_I\otimes \O_J\otimes
			\O(-L_j)\to T_I\otimes \O_J$ is injective  
		\item \label{cond2} $\Ext^i_{\O_{I\cup J}}(T_I\otimes \O_J,
			T_{I\cup J})=0$ for all $i>0$.
	\end{enumerate} then 
	\[
		\bigoplus_{I\subseteq \{1,\dots,n\}}\left(\bigotimes_{i\not\in I}
	A_if_i	
	\otimes
		\bigotimes_{i\in I}
		\frac{A_i}{\langle e_i\rangle}
		\otimes T_I\right)
	\] is tilting in $\A$.
\end{theorem}

\section{Examples}\label{sec:applications}
We now apply Theorem \ref{Thm:order} to various situations. Note that if $T_I$ is in fact a tilting bundle
then Condition (1) of the theorem is automatically satisfied. Furthermore, by Serre vanishing, we can always twist
$T_I$ so that Condition (2) is also satisfied. 
\subsection{Weighted projective lines}Let $X=\P^1_{X_0:X_1}$ and $\A=\coh X$. For $i=1,\dots,n$
choose points $L_i=(\lambda_{0,i}:\lambda_{1,i})$ %which we view as as elements of $\Hom(\O(-1),\O)$ 
and corresponding weights $p_i$.
%Let $F_i=-\otimes_X \O(-1)$ and 
%\[\eta_i\colon F\F=\F\otimes\O(-1)\xrightarrow{\id\otimes L_i}\F\otimes\O\to \F. \]
$T_\emptyset=\O_X\oplus \O_X(1)$ is a tilting object
in $\coh X$ and $T_{\{i\}}=\O_{L_i}$ is a tilting object in $\coh L_i$. Then
\[T=\bigotimes_{i=1}^n A_if_i\otimes T_\emptyset\bigoplus_{i=1}^n\left( \bigotimes_{j\neq i}A_jf_j\otimes \frac{A_i}{\langle
e_i\rangle}\otimes T_{\{i\}} \right)\]
is a tilting object in
\[\mod\bigotimes_{i=1}^{n}A_i:=\mod\bigotimes_{i=1}^{n}T_{p_i}(\O,\O(-L_i))\simeq \AAA \] with endomorphism algebra
\[\xymatrix{
	&&\cdot\ar[r]&\dots
		\ar[r]&
		\cdot\\
		&&\cdot\ar[r]&\dots
		\ar[r]&
		\cdot\\
		\ar@<0.5ex>[r]^{X_0}
		\cdot\ar@<-0.5ex>[r]_{X_1}&
		\ar[uur]^{y_1}
		\cdot\ar[ur]_{y_2}\ar[dr]_{y_n}
		&\vdots&\vdots&\vdots\\
		&&\cdot\ar[r]&\dots
		\ar[r]&
	\cdot}\] with relations \[(\lambda_{1,i}X_0-\lambda_{0,i}X_1)y_i=0.\]
	This algebra is known as the ``squid''.
%\[\xymatrix{
%	&&\Rpi\iota\Delta_1\O_{p_1}\ar[r]&\dots
%		\ar[r]&
%		\Rpi\iota\Delta_{p_1}\O_{p_1}\\
%		&&\Rpi\iota\Delta_1\O_{p_2}\ar[r]&\dots
%		\ar[r]&
%		\Rpi\iota\Delta_{p_2}\O_{p_2}\\
%		\Rpi^n\O\ar@<0.5ex>[r]^{T_0}\ar@<-0.5ex>[r]_{T_1}&
%		\Rpi^n\O(1)\ar[uur]^{y}\ar[ur]_{y}\ar[dr]_{y}
%		&\vdots&\vdots&\vdots\\
%		&&\Rpi\iota\Delta_1\O_{p_n}\ar[r]&\dots
%		\ar[r]&
%		\Rpi\iota\Delta_{p_n}\O_{p_n}}\]

\subsection{Weighted $\P^2$}
Let $\A=\coh \P^2$.
Note that both lines and smooth conics in $\P^2$ are isomorphic to $\P^1$ and hence have tilting bundles.
For $i=1,\dots, l$ fix hyperplanes $L_i$ as well as for $i=l+1,\dots,n$ smooth conics $L_i$ and weights $p_i$.
As before, we consider the category
\[\A\simeq \mod \bigotimes_{i=1}^{n}T_{p_i}(\O,\O(-L_i)). \]
$T_\emptyset=\O_{\P^2}(-2) \oplus \O_{\P^2}(-1) \oplus \O_{\P^2}$ is a tilting bundle in $\coh \P^2$, 
$T_{\{i\}}=\O_{L_i}\oplus \O_{L_i}(1)$ is a tilting bundle in $\coh L_i$ and $T_{\{i,j\}}=\O_{L_i\cap L_j}$ is a tilting
bundle in $L_i\cap L_j$. 
Thus, by (\ref{Thm:order})
\[T=\bigotimes_{i=1}^n A_if_i\otimes T_\emptyset\bigoplus_{i=1}^n\left( \bigotimes_{j\neq i}A_jf_j\otimes \frac{A_i}{\langle
e_i\rangle}\otimes T_{\{i\}}\right)
\bigoplus_{i,j=1}^{n}\left( \bigotimes_{h\not\in\{i,j\}}A_hf_h\otimes\frac{A_i}{\langle e_i\rangle}\otimes\frac{A_j}{\langle
e_j\rangle}\otimes T_{\{i,j\}} 
\right)\]
%\[T=\Rpi^nT_{\emptyset}\bigoplus_{i\in[1,n]}\Rpi^{n-1}\iota
%T_{\left\{ i \right\}}^{\left\{ i \right\}}\bigoplus
%_{i,j}\Rpi^{n-2}\iota^2T_{\left\{ i,j \right\}}^{\left\{ i,j
%\right\}}\] 
is a tilting object provided the two conditions of the theorem are satisfied.
Condition (\ref{cond1}) is automatically satisfied as all the titling objects are in fact vector bundles.
Furthermore,  Condition (\ref{cond2}) is also satisfied as \[\Ext^1_{L_i}(T_\emptyset \otimes \O_{L_i}, T_{\{i\}})=\left\{  
	\begin{array}[]{ll}
		\Ext^1_{L_i}(\O_{L_i}(-2)\oplus \O_{L_i}(-1)\oplus \O_{L_i}, \O_{L_i}\oplus\O_{L_i}(1))=0,\; i=0,\dots, l\\
		\Ext^1_{L_i}(\O_{L_i}(-4)\oplus \O_{L_i}(-2)\oplus \O_{L_i}, \O_{L_i}\oplus\O_{L_i}(1))=0,\; i=l+1,\dots,n
	\end{array}
\right.\] and so $T$ is indeed a tilting bundle in $\mod A$.
\subsection{Weighted Hirzebruch surfaces}
In this section we following King's conventions from \cite{K}. For $m\ge 0$ the Hirzebruch surface is defined as
\[\Sigma_m=\P(\O_{\P^1}(-m)\oplus \O_{\P^1}).\]In this section $\O=\O_{\Sigma_m}$. It is well know that $\Pic \Sigma_m=\Z^2$ with intersection form
$
\begin{bmatrix}
	0&1\\1&m
\end{bmatrix}
$ and canonical bundle $\O(m-2,-2)$. 
Using the adjunction formula, which states that a smooth genus $g$ curve $C$ on a surface $X$ with
canonical divisor $K$ satisfies \[2g-2=C.(C+K),\] we see that any smooth curve of type $(a, 1)$ or $(1,0)$ is rational
and hence has a tilting bundle.

Similar to the weighted $\P^2$ case, for $i=1,\dots,l$ let $L_i$ be a smooth $(a_i,1)$ divisor and for $i=l+1,\dots,n$ a $(1,0)$ divisor.
As before, we consider the category
\[\A\simeq \mod\bigotimes_{i=1}^{n}T_{p_i}(\O,\O(-L_i)). \]
Let, $T_\emptyset = \O\oplus\O(1,0)\oplus \O(0,1)\oplus\O(1,1)$ which is a tilting bundle in $\coh \Sigma_m$. For $i=1,\dots,l$ let
$T_{\{i\}}=\O_{L_i}(a_i+m)\oplus \O_{L_i}(a_i+m+1)$ and $T_{\{i\}}=\O_{L_i}\oplus\O_{L_i}(1)$ for $i=l+1,\dots,m$. Finally let
$T_{\{i,j\}}=\O_{L_i\cap L_j}$.
Thus, by (\ref{Thm:order})
\[T=\bigotimes_{i=1}^n A_if_i\otimes T_\emptyset\otimes\bigoplus_{i=1}^n\left( \bigotimes_{j\neq i}A_jf_j\otimes \frac{A_i}{\langle
e_i\rangle}\otimes T_{\{i\}}\right)\otimes
\bigoplus_{i,j=1}^{n}\left( \bigotimes_{h\not\in\{i,j\}}A_hf_h\otimes\frac{A_i}{\langle e_i\rangle}\otimes\frac{A_j}{\langle
e_j\rangle}\otimes T_{\{i,j\}} 
\right)\]
%\[T=\Rpi^nT_{\emptyset}\bigoplus_{i\in[1,n]}\Rpi^{n-1}\iota
%T_{\left\{ i \right\}}^{\left\{ i \right\}}\bigoplus
%_{i,j}\Rpi^{n-2}\iota^2T_{\left\{ i,j \right\}}^{\left\{ i,j
%\right\}}\] 
is a tilting object provided the two conditions of the theorem are satisfied.
Condition (\ref{cond1}) is automatically satisfied as all the titling objects are in fact vector bundles.
Furthermore,  Condition (\ref{cond2}) is also satisfied as 
\begin{align*}
&\Ext^1_{L_i}(T_\emptyset \otimes \O_{L_i}, T_{\{i\}})\\&=\left\{  
	\begin{array}[]{ll}
		\Ext^1_{L_i}(\O_{L_i}\oplus \O_{L_i}(1)\oplus \O_{L_i}(a_i+m)\oplus\O_{L_i}(a_i+m+1), \O_{L_i}(a_i+m)\oplus\O_{L_i}(a_i+m+1))=0,& i=0,\dots l\\
		\Ext^1_{L_i}(\O\oplus \O_{L_i}\oplus \O_{L_i}(1)\oplus \O_{L_i}(1), \O_{L_i}\oplus\O_{L_i}(1))=0,\quad i=l+1,\dots,n
	\end{array}
\right.
\end{align*}
and so $T$ is indeed a tilting bundle in $\mod A$.
\subsection{Squids}
We have already seen the squid algebra which arose as the endomorphism algebra of tilting objects on weighted projective lines. We now
generalise this to higher dimensional weighted projective spaces.

Let $X=\P^d_{X_0:\dots:X_d}$ and for $i=1,\dots,n$ let $L_i\colon \ell_i(X_0,\dots,X_d)=0$ be hyperplanes in general position. 
For $I\subseteq \{1,\dots,n\}$ with $|I|\le d$ 
\[ T_I=\O_I(|I|)\oplus \O(|I|+1)\oplus \dots\oplus\O_I(d)\] is a tilting bundle in $\coh L_I$ where 
\[L_I=\bigcap_{i\in I}L_i\] Furthermore, as we have seen the category $\mod A$ where \[A=\bigotimes_{i=1}^nT_{p_i}(\O_X,\O_X(-L_i))\] has a
tilting as described in $(\ref{Thm:order})$ since Condition (\ref{cond1}) is trivial as all $T_I$ are bundles and Condition (\ref{cond2}) is
easy to verify with our choice of $T_I$.

We now describe $\End_AT$ where $T$ is given by (\ref{Thm:order}) presenting it as a quiver with relations. 
First, we will describe the vertices, then the arrows and finally the
relations. For simplicity, we will allow non-admissible relations.

Vertices: Let $S$ be as before. For each $\a\in S$, let $I_\a=\left\{ i \;|\;a_i\neq 1\right\}$.
The vertices are labelled 
\[\O_{\a}(|I_\a|+i)\]  with $\a\in S$ and $0\le i\le d-|I_\a|$.

Arrows:
\begin{itemize}
	\item $(d+1)$ arrows labelled $X_{I_\a}^0,\dots,X_{I_\a}^d$ between
		\[\xymatrix{
			\O_{\a}(|I_\a|+i)\ar@<1.5ex>[r]^{X^0_{I_\a}}\ar@<-2.5ex>[r]_{X^d_{I_\a}}\ar@{}@<0.3ex>[r]|{\vdots}&\O_{\a}(|I_\a|+i+1)
		}\] whenever $i\leq d-|I_\a|$.
	\item $1$ arrow labelled $x$ between \[\O_{\a}(|I_\a|+i)\xrightarrow{x}\O_{{\hat \a}}(|I_{\hat\a}|+i)\] where
		$\hat\a=\a+\e_j$ for all $j$ with $1<a_j\le p_j$. Note that in this case $I_\a=I_{\hat\a}$.
	\item $1$ arrow labelled $y_j$ between \[\O_{\a}(|I_\a|+i)\xrightarrow{y_j}\O_{{\hat \a}}(|I_{\hat\a}|+i-1)\] where
		$\hat\a=\a+\e_j$ for all $j$ with $a_j=1$. Note that in this case $|I_\a|=|I_{\hat\a}|-1$.
\end{itemize}
Relations:
\begin{itemize}
	\item Commutativity relations: $X^i_{I_\a}X^j_{I_\a}=X^j_{I_\a}X^i_{I_\a},\;X^i_{I_\a}y_j=y_jX^i_{I_{\hat \a}},\; y_iy_j=y_jy_i$ and $X^i_{I_\a}x=xX^i_{I_\a}$.
	\item $\ell_i(X^0_{I_\a},\dots,X^d_{I_\a})=0$ for all $\a\in S$ and $i\in I_{\a}$.
	\item For all $\a\in S$ with $a_j=1$ and subquivers of the form
	\[\xymatrix{
			\O_{\a}(|I_\a|+i)\ar@<1.5ex>[r]^{X^0_{I_\a}}\ar@<-2.5ex>[r]_{X^d_{I_\a}}\ar@{}@<0.3ex>[r]|{\vdots}&
			\O_{\a}(|I_\a|+i+1)\ar[r]^{y_j}&\O_{{\hat \a}}(|I_{\hat\a}|+i)
		}\] where $\hat \a=\a+\e_j$ with we have $1$ (extra) relation \[\ell_j(X^0_{I_\a},\dots,X^d_{I_\a})y_j=0.\]
\end{itemize}

\begin{example}[On $\P^2$, $2$ weights, both of weight $3$] Consider $\P^2_{X_0:X_1:X_2}$ and hyperplanes $L_i\colon \ell_i(X_0,X_1,X_2)=0$
	for $i=1,2$. Let \[A=
		\begin{bmatrix}
			\O&\O(-L_1)&\O(-L_1)\\\O&\O&\O(-L_1)\\\O&\O&\O
		\end{bmatrix}\otimes
		\begin{bmatrix}
			\O&\O(-L_2)&\O(-L_2)\\\O&\O&\O(-L_2)\\\O&\O&\O
		\end{bmatrix}
		=A_1\otimes A_2
	\]
	Then $T_\emptyset =\O\oplus\O(1)\oplus \O(2)$, $T_{\{i\}}=\O_{L_i}(1)\oplus \O_{L_i}(2)$ and
	$T_{\{1,2\}}=\O_{L_1\cap L_2}$. Then by (\ref{Thm:order})
	\[
	T=	\left(A_1f_1\otimes A_2f_2\otimes
		T_{\emptyset}\right)\oplus
		\left(\frac{A_1}{\langle e_1\rangle}\otimes
		A_2f_2\otimes T_{1}\right)\oplus \left(A_1f_1\otimes
		\frac{A_2}{\langle e_2\rangle}\otimes T_{2}\right)\oplus
		\left(\frac{A_1}{\langle e_1\rangle}\otimes
		\frac{A_2}{\langle e_2\rangle }\otimes
		T_{1,2}\right)
	\] is a tilting object in $\A=\mod A$.
$\End_AT$ is given by the following quiver
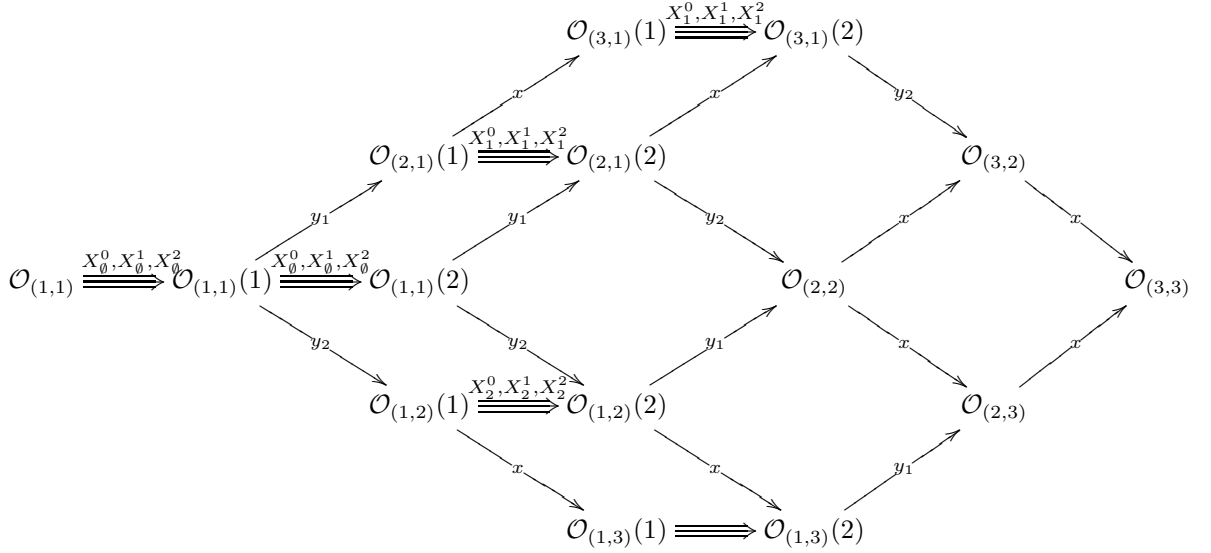
\[\xymatrix@R=30pt@C=30pt{
	&&&\O_{(3,1)}(1)\ar@3[r]^{X^0_1, X^1_1, X^2_1}&\O_{(3,1)}(2)\ar[dr]|{y_2}\\
	&&\O_{(2,1)}(1)\ar@3[r]^{X^0_1, X^1_1, X^2_1}\ar[ur]|x&\O_{(2,1)}(2)\ar[dr]|{y_2}\ar[ur]|x&&\O_{(3,2)}\ar[dr]|x\\
	\O_{(1,1)}\ar@3[r]^{X^0_\emptyset, X^1_\emptyset, X^2_\emptyset}& \O_{(1,1)}(1)\ar@3[r]^{X^0_\emptyset, X^1_\emptyset, X^2_\emptyset}\ar[ur]|{y_1}\ar[dr]|{y_2}&
	\O_{(1,1)}(2)\ar[ur]|{y_1}\ar[dr]|{y_2}&&\O_{(2,2)}\ar[ur]|x\ar[dr]|x&&\O_{(3,3)}\\
	&&\O_{(1,2)}(1)\ar@3[r]^{X^0_2, X^1_2, X^2_2}\ar[dr]|x&\O_{(1,2)}(2)\ar[ur]|{y_1}\ar[dr]|x&&\O_{(2,3)}\ar[ur]|x\\
	&&&\O_{(1,3)}(1)\ar@3[r]&\O_{(1,3)}(2)\ar[ur]|{y_1}
}\] with commutativity relations as well as:
\[
	\ell_1(X^0_\emptyset, X^1_\emptyset, X^2_\emptyset)y_1,\;\ell_2(X^0_\emptyset, X^1_\emptyset, X^2_\emptyset)y_2,\;
		\ell_1(X^0_1, X^1_1, X^2_1),\;	\ell_2(X^0_2, X^1_2, X^2_2), \ell_2(X^0_1, X^1_1, X^2_1)y_2, \ell_1(X^0_2, X^1_2, X^2_2)y_1
\]
\end{example}


\begin{thebibliography}{10}
	\bibitem[AdJ]{AdJ} M, Artin; J. de Jong, \emph{Stable orders on surfaces} (manuscript in preparation).
	\bibitem[AR]{AR} Adams, William W.; Rieffel, Marc A., \emph{Adjoint functors and derived
		functors with an application to the cohomology of semigroups}. J. Algebra 7 1967
		25--34. 	
	\bibitem[BBD]{BBD} A. Beĭlinson; J. Bernstein; P. Deligne, \emph{Faisceaux pervers}. (French) [Perverse sheaves].
		Analysis and topology on singular spaces, I (Luminy, 1981), 5–171, Astérisque, 100, Soc. Math. France, Paris, 1982.
\bibitem[CI]{CI} D. Chan; C. Ingalls, \emph{Non-commutative coordinate rings and stacks}. Proc. London Math. Soc. (3) 88 (2004), no. 1, 6388		
\bibitem[GL]{GL} W. Geigle; H. Lenzing, \emph{A class of weighted projective curves arising in representation theory of finite-dimensional algebras. Singularities, representation of algebras, and vector bundles} (Lambrecht, 1985), Lecture Notes in Math., 1273, Springer, Berlin, 1987.
\bibitem[HIMO]{HIMO} M. Herschend, O. Iyama, H. Minamoto, S. Oppermann, \emph{Geigle-Lenzing spaces
	and canonical algebras in dimension $d$}, arXiv:1409.0668
\bibitem[IL]{IL} O. Iyama, B. Lerner, \emph{Tilting bundles on orders on $\P^d$}, arXiv:1306.5867
\bibitem[K]{K} A. King, \emph{Tilting bundles on some rational surfaces}, 2007. 
\bibitem[ML]{ML} S. Mac Lane, \emph{Homology}, Die Grundlehren der mathematischen Wissenschaften, Bd. 114 Academic Press, Inc., Publishers, New York; Springer-Verlag, Berlin-Göttingen-Heidelberg 1963 


\bibitem[RVdB]{RVdB} I. Reiten; M. Van den Bergh, \emph{Grothendieck groups and tilting objects}. Special issue dedicated to Klaus Roggenkamp on the occasion of his 60th birthday.

\end{thebibliography}
\end{document}